\documentclass[12pt]{amsart}

\usepackage[colorlinks=true,linkcolor=blue, allcolors=blue]{hyperref}
\usepackage[backend=biber, style=alphabetic, isbn=false, url=false, doi=false]{biblatex}
\makeglossary
\usepackage{fullpage}
\usepackage{amsmath,amssymb,amsfonts,amsthm}
\usepackage{tikz}
\usepackage{enumitem}
\usepackage{graphicx}
\usepackage{pgfplots}
\usetikzlibrary{arrows}
\usepackage{thm-restate}
\usepackage{cleveref}
\usepackage{MnSymbol}
\usepackage{csquotes}
\usepackage{enumitem}
\usepackage{verbatim}
\usepackage{tikz}
\usepackage{tikz-cd}
\usepackage{color}
\usepackage{graphicx}

\newcommand\Laf{\mathrm{Laf}}
\newcommand\iLaf{\mathrm{iLaf}}
\newcommand{\Spec}{\mathrm{Spec}}
\newcommand{\Hilb}{\mathrm{Hilb}}
\newcommand{\GL}{\mathrm{GL}}
\newcommand{\rad}{\mathrm{rad}}
\newcommand{\sur}{\twoheadrightarrow}

\newcommand\cT{\mathcal{T}}

\newcommand\tr{\mathrm{tr}}
\newcommand{\Sets}{\mathrm{Sets}}
\newcommand{\amo}{\stackrel{a}{\sim}}
\newcommand{\en}{\textrm{End}}

\newcommand{\fs}{\mathfrak{s}}

\newcommand{\ir}{\mathrm{Irr}}

\newcommand{\Hom}{\mathrm{Hom}}
\newcommand{\ed}{\mathrm{End}}
\newcommand{\Sym}{\mathrm{Sym}}
\newcommand{\fR}{\mathfrak{R}}
\newcommand{\bz}{\vec{\mathbf{z}}}

\newcommand{\cs}{\mathbf{sc}}

\newcommand{\bv}{\Omega}

\newcommand{\mm}{\mathcal{M}}
\newcommand{\oo}{\mathcal{O}}
\newcommand{\om}{\Omega}
\newcommand{\f}{\mathcal{F}}
\newcommand{\fB}{\mathfrak{B}}

\newcommand{\bG}{\mathbf{G}}

\newcommand{\bP}{\mathbf{P}}

\newcommand{\cH}{\mathcal{H}}
\newcommand{\cE}{\mathcal{E}}
\newcommand{\cO}{\mathcal{O}}
\newcommand{\cR}{\mathcal{R}}
\newcommand{\cQ}{\mathcal{Q}}
\newcommand{\cG}{\mathcal{G}}

\newcommand{\bA}{\mathbb{A}}
\newcommand{\Z}{\mathbb{Z}}
\newcommand{\N}{\mathbb{N}}
\newcommand{\C}{\mathbb{C}}

\newcommand{\R}{\mathbb{R}}

\newcommand{\git}{\mathbin{/\mkern-5mu/}}
\newcommand{\nex}{\bigwedge^{n}}
\newcommand{\cha}{\mathrm{char}}
\newcommand{\hc}{\cH_{\fs}}

\newcommand{\un}{\mathcal{Q}}

\newcommand{\Quot}{\mathrm{Quot}}
\newcommand{\maps}{\rightarrow}

\newtheoremstyle{mybold}
{3pt}
{3pt}
{}
{0pt}
{\bfseries}
{.}
{.5em}
{}

\newtheorem{theorem}{Theorem}[section]
\newtheorem{lemma}[theorem]{Lemma}
\newtheorem{definition}[theorem]{Definition}

\newtheorem{proposition}[theorem]{Proposition}
\newtheorem{example}[theorem]{Example}
\newtheorem{corollary}[theorem]{Corollary}

\theoremstyle{mybold}
\newtheorem{remark}[theorem]{Remark}

\title{The Lafforgue variety and irreducibility of induced representations}
\author{Kostas I. Psaromiligkos}
\address{UFR de Mathématiques, Université Clermont Auvergne, 3 place Vasarely, Clermont-Ferrand, 63170 Aubière, France. email:
	konstantinos.psaromiligkos@uca.fr.}
\date{\today} 

\begin{document}

	\begin{abstract} We construct the Lafforgue variety, an affine scheme equipped with an open dense subscheme parametrizing the simple modules of a non-commutative unital algebra $R$ over any field $k$, provided that the center $Z(R)$ is finitely generated and $R$ is finitely generated as a $Z(R)$-module. Our main technical tool is a generalization of the Hilbert scheme for non-commutative algebras, which may be of independent interest.
		
	Applying our construction in the case of Hecke algebras of Bernstein components, we derive a characterization for the irreducibility of induced representations in terms of the vanishing of a generalized discriminant on the Bernstein variety. We explicitly compute the discriminant in the case of an Iwahori-Hecke algebra of a split reductive $p$-adic group. 
	\end{abstract} 
	\maketitle 
	\tableofcontents
	
	\section{Introduction}
	\subsection{Bernstein theory}
	
	The category of smooth representations $\mm(G)$ of a reductive $p$-adic group $G$ is not semisimple. Instead, there is a splitting of $\mm(G)$ as a direct product of indecomposable categories by virtue of the Bernstein decomposition theorem  	
	$$\mm(G)\cong \prod_{\fs\in \fB(G)} \mm_{\fs}(G),$$ indexed by the set of connected components $\fB(G)$ of the Bernstein variety $\om(G)=\bigsqcup_{\fs\in \fB(G)} \om_{\fs}(G),$ see \cite[Proposition 2.10]{Bernstein1984}.
	
	In particular, to each smooth irreducible representation $\pi\in \ir(G)$ we can attach uniquely up to $G$-conjugation its cuspidal support $\cs(\pi):=(M,\sigma)_G$, where $M$ is a Levi subgroup of $G$ and $\sigma$ a supercuspidal irreducible representation of $M$ such that $\pi$ embeds in the parabolic induction $i_M^G(\sigma).$ Then, $\om(G)$ parametrizes the set of conjugation classes of possible cuspidal supports. As $i_M^G(\sigma)$ is of finite length, $\om(G)$ is also a finite-to-one parametrizing space for $\ir(G).$  We define $\ir^{\fs}(G)\subseteq \ir(G)$ to be the subset of irreducible representations with cuspidal support in $\fs$. The splitting provided by Bernstein's decomposition theorem induces the partition 
	\begin{equation}\label{berpart}
		\ir(G)=\bigsqcup_{\fs\in \fB(G)} \ir^{\fs}(G)
	\end{equation}
	on the level of irreducible objects.
	
	With the intent of studying $\mm(G)$, various versions of Hecke algebras have been introduced, with the property that their module category is equivalent to some subcategory of $\mm(G)$. The centers of Hecke algebras are isomorphic to subrings of the ring of regular functions on $\om(G).$ The study of the representation theory of $p$-adic reductive groups via Hecke algebras has provided an abundance of beautiful results, see for example \cite{Iwahori1965}, \cite{Kazhdan1987}, \cite{Aubert2021}.
	
	\subsection{The Lafforgue variety}
	
	In \cite{Lafforgue2016}, Laurent Lafforgue predicted the existence of an algebraic variety classifying smooth irreducible representations of a $p$-adic reductive group with regular functions being generated by traces over Hecke algebras.  
	
	More precisely, Lafforgue's proposed construction works as follows. For every smooth representation $(V,\pi)$ of $G$ we have $V=\displaystyle\bigcup\limits_{K} V^K$ where $K$ ranges over all compact open subgroups $K\leq G$ and $V^K$ denotes the subspace of $K$-fixed vectors. As a consequence of Bernstein's admissibility theorem, if $V$ is also irreducible, $V^K$ is finite dimensional. Let $\mm_K(G)$ be the full subcategory of smooth representations generated by $K$-fixed vectors and $\cH_K(G)$ the Hecke algebra of $K$-biinvariant locally constant compactly supported distributions on $G$ - see subsection 2.1. We have
	$$\mm_K(G)\cong \mm(\cH_K(G))$$
	where $\mm(\cH_K(G))$ denotes the module category of $\cH_K(G)$. Therefore, $\ir_K(G)\cong \ir(\cH_K(G)).$
	
	We consider for each $r\in \cH_K(G)$ a function $f_r: \ir(\cH_K(G))\rightarrow \C$ defined by $f_r(V)=\tr_V(r)$. Let $\cT_K$ be the ring of functions on $\ir(\cH_K(G))$ generated by all $f_r, r\in \cH_K(G)$, which we call the \textit{ring of traces}. The set $\ir(\cH_K(G))$ can be naturally embedded in $\Spec(\cT_K)$, since any $V\in \ir(\cH_K(G))$ gives a geometric point via the evaluation homomorphism. Lafforgue predicted that $\ir_K(G)\cong \ir(\cH_K(G))$ embeds as an open dense subscheme of $\Spec(\cT_K).$
	
	We prove the existence of this basic object in the following more general setting: Let $R$ be an associative unital $k$-algebra over any field $k$, such that the center $Z(R)$ is finitely generated and $R$ is finite as a $Z(R)$-module. Let $A$ be any subalgebra of $Z(R)$ such that $R$ is a finite $A$-module. An irreducible $R$-module will be finite dimensional by an argument similar to \cite[Theorem A.4]{Jantzen2002}. Therefore, if $\cha(k)=0$, by the same procedure we can define the ring of traces $T_R$ of $R$ over $A$. We define $\Laf_{R/A}:=\Spec(T_R)$ to be the \textit{Lafforgue variety}.	
	
	The subalgebra $A\subseteq Z(R)$, acts by a character on any simple module, and thus we get a natural projection $\Laf_{R/A}\rightarrow \Spec(A)$. 
	
	We can now state our main theorem, where we also remove any assumptions on $k$. Notice that our definition of the ring $T_R$ will be more involved in the case $\cha(k)>0$ - see subsection 3.3.
	
	\begin{restatable}{theorem}{main}\label{main}
		$\ir(R)$ forms the set of $\bar{k}$-points of a dense Zariski open subscheme $\iLaf_{R/A}$ of $\Laf_{R/A}.$ The projection $p: \Laf_{R/A}\rightarrow \mathrm{Spec}(A)$ is finite.
	\end{restatable}
	
	The main difficulty in proving Theorem \ref{main} is that Lafforgue's proposed construction of the trace ring $T_R$ does not provide much information on its structure. What we thus need is a framework in algebraic geometry which gives rise to a more workable definition of the trace ring. This framework turns out to be a non-commutative generalization of the classical Hilbert-Chow morphism.
	
	Our strategy is as follows: we will first construct a non-commutative Hilbert scheme for the finite $A$-algebra $R$, which is a proper $A$-scheme $Q$. Next, using the trace (or the determinant in the positive characteristic case), we construct a morphism from the Hilbert scheme to an affine $A$-scheme $V$. The morphism $Q\to V$ can then be shown to factor through a closed subscheme of $V$ which is finite over $\Spec(A)$, of which the coordinate ring can be identified with $T_R$ in the case $\cha(k)=0$. 
	
	\subsection{The nested Quot scheme}
	
	To construct the non-commutative and nested non-commutative Hilbert schemes that we will need, which we believe may be of independent interest, we also need a nested version of the Quot functor $\mathrm{F}_n\Quot_{E,\bar{P}(t)}^{X/S}:(Sch/S)\rightarrow Sets$, see Definition \ref{nquotdef}.
	
	The representability of the Quot functor parametrizing flat families of quotients of a coherent sheaf, was proven in Grothendieck's seminal paper \cite{Grothendieck1961}. It is the basic building block for a lot of constructions in moduli space theory and many classical constructions as the Grassmannian and the Hilbert scheme are special cases of Quot schemes. For an exposition, we recommend Sernesi's book \cite[\S 4]{Sernesi2006}. We use the Quot scheme to construct the non-commutative Hilbert scheme. 
	
	In \cite[\S 4.5]{Sernesi2006}, the classical nested Hilbert functor and various versions of it are also proven to be representable. In \cite[Remark 4.5.4 (iv)]{Sernesi2006}, it is mentioned that a nested version of the Quot functor is also representable, but the result seems to not have appeared in the literature.
	
	While we found specific cases in the literature, as in \cite[Theorem 1]{CiocanFontanine1995}, \cite[Theorem 1]{Kim1995}, or more recently \cite[\S 1.1]{Monavari2022}, we could not find a reference for the general case.
	
	Since the existence of this basic object is important for our construction of the nested non-commutative Hilbert scheme, we prove in the Appendix that the nested Quot functor, see Definition \ref{nquotdef}, is also representable in the general case, following Sernesi's exposition. The following is Theorem \ref{nquot}.
	
	\begin{theorem}\label{nquoti}
		The nested Quot functor $\mathrm{F}_n\Quot_{E,\bar{P}(t)}^{X/S}:(Sch/S)\rightarrow Sets$ parametrizing flat families of nested quotients of $E$ of length $n$, is representable by a projective scheme over $S$.
	\end{theorem} 
	
	\subsection{Irreducibility of induced representations}
	
	We apply our results to the question of irreducibility for a parabolically induced representation from a cuspidal datum, a long studied subject, see for example \cite{Bernstein1976}, \cite{Muller1979}, \cite{Kato1981}, \cite{Kazhdan1987}. In particular, we consider the Hecke algebra $\cH_{\fs}(G)$ of a Bernstein component $\fs\in \fB(G)$, with the property that its module category $\mm(\hc(G))$ satisfies $$\mm_{\fs}(G)\cong \mm(\cH_{\fs}(G)).$$
	In \cite{Solleveld2022}, it was shown that $\cH_{\fs}$ is almost Morita equivalent, ie. the categories of finite-dimensional modules are equivalent, to a twisted affine Hecke algebra $H_{\fs}.$ 
	
	Applying Theorem \ref{main} to the case of the Hecke algebra $\cH_{\fs}$ of a Bernstein component $\fs \in \fB(G)$ and its center $Z_{\fs}$, we get a finite map 
	$$p: \Laf_{\cH_{\fs}/Z_{\fs}}\rightarrow \om_{\fs}$$ 
	from the Lafforgue variety to the Bernstein variety. When restricted to $\iLaf_{\cH_{\fs}/Z_{\fs}},$ it agrees with Bernstein's cuspidal support map upon identifying $\ir^{\fs}(G)$ with $\iLaf_{\cH_{\fs}/Z_{\fs}}(\C)$.
	
	Returning to the general case of any algebra $R$ satisfying our condition and a central subalgebra $A$, we stratify $\Spec(A)$ according to the cardinality of the fibers of $p$.  If $R=\hc$, $A=Z_{\fs}$, the parabolic induction $i_M^G(\sigma)$ from a cuspidal datum $(M,\sigma)_G\in \Spec(Z_{\fs})$ is irreducible if and only if $|p^{-1}(M,\sigma)|=1,$ ie. on the open dense stratum $X_0$ of the cardinality stratification.
	
	If $A$ is regular over a field $k$ of characteristic 0 and $R$ is also a locally free $A$-module, we have another concrete description of the stratification. Fixing a central character $\chi:A\rightarrow k$, a simple $R$-module with central character $\chi$ corresponds to an $R_{\chi}=\left(R\otimes_{A,\chi}k\right)$-module. We can then describe the stratification by studying the rank of the Jacobson radical of $R_{\chi}$. 
	
	We mainly apply this result to the open dense stratum $X_0$. We define a notion of a \textit{generalized discriminant} $d_{R/A}$, which is a principal ideal of $A$, with the property that the complement of its zero set in $\Spec A$ is $X_0$.	
	
	Induced representations are irreducible for generic cuspidal data, so they are irreducible exactly on $X_0$. We choose a regular central subalgebra $A\subseteq Z_{\fs}$ with $f:\Spec(Z_{\fs})\rightarrow \Spec(A)$ finite. Using the Jacobson stratification previously described, we prove the following.
	
	\begin{restatable}{theorem}{ind}\label{ind}
		Let $(M,\sigma)_G$ be a cuspidal datum. Then, $i_M^G(\sigma)$ is irreducible if and only if $(M,\sigma)\in X_0$. Outside of the singular locus $Z(f)$, this is equivalent to 
		$$d_{\cH_{\fs}/A}(f(M,\sigma))\neq 0.$$
	\end{restatable}
	
	We develop computational methods for generalized discriminants that work well with explicit presentations as in Solleveld's theorem \cite{Solleveld2022}. In particular, we explicitly calculate the discriminant for the unramified Bernstein component and the Iwahori-Hecke algebra of a split reductive $p$-adic group to retrieve results about irreducibility of principal series appearing, for example, in \cite{Kato1981}.

	\subsection{Outline}
	
	Section 2 is of expository nature and is only important for the applications. The reader interested only in the main theorem can skip it freely. We recall various versions of Hecke algebras, both to amend possible confusion stemming from the existence of multiple algebras going by that name in the literature and to make our treatment more self-contained.
	
	Section 3 is the main technical heart of the paper. We define our non-commutative generalizations of the Hilbert scheme and the Hilbert-Chow morphism. We use a trace map and carry out the proof of Theorem \ref{main} in the characteristic zero case. Then, we use the determinant map to treat the positive characteristic case. We also show the Lafforgue variety construction is independent of the auxiliary choice of a central subalgebra $A\subseteq Z(R)$. 
	
	In Section 4, assuming the algebra $R$ is Cohen-Macaulay and $k$ is characteristic $0$, we construct the Jacobson stratification with the purpose of studying explicitly the structure of the Lafforgue variety.
	
	In Section 5, we define generalized discriminants and study the case of Hecke algebras to prove Theorem \ref{ind}. We prove properties of discriminants to make them more amenable to calculation, including a generalization of the classical behavior of the discriminant in a tower of extensions of number rings to general commutative algebras, which doesn't seem to appear in the literature for our case, using the generalized Riemann-Hurwitz formula. In particular, we show the following.
	
	\begin{restatable}{lemma}{transd}\label{transd}
		For a tower of extensions $C/B/A$ such that $A,B$ are commutative and regular, $C$ is commutative, $C$ is free of rank $n$ as a $B$-module and $B$ is free as an $A$-module, we have that
		$$d_{C/A}=(d_{B/A})^n\cdot N_{B/A}(d_{C/B}),$$
		where $N_{B/A}$ is the norm function. 
	\end{restatable}
	
	As an application, we compute the discriminant for the case of an Iwahori-Hecke algebra of a split reductive $p$-adic group.
	
	In the Appendix, we recall the definitions of the Hilbert and Quot schemes, and prove the representability of the nested Quot functor.
	
	\subsection{Acknowledgements} 
 	
 	I warmly thank Anne-Marie Aubert, Micah Gay, Mathilde
 	Gerbelli-Gauthier, Philippe Gille, Aristides Kontogeorgis, Zhilin Luo, Siddharth Mahendraker, Benedict
 	Morrissey, Bao Ch\^au Ng\^o, Simon Riche, Minh-Tam Quang Trinh, Yiannis Sakellaridis, Maarten Solleveld
 	and Griffin Wang for their interest, comments and many helpful discussions.
 	
 	This work was part of my PhD thesis at the University of Chicago, advised by Bao Ch\^au
 	Ng\^o, to whom I am thankful for his mentorship and guidance throughout the process, as
 	well as suggesting this problem.
 	
 	This project has received funding from the European Research Council (ERC) under the European Union’s Horizon 2020 research and innovation programme (grant agreement No. 101002592). During my PhD, this work was also supported by a Graduate Research Fellowship from the Onassis foundation.
	
	\section{Hecke algebras}
	
	Let $\bG$ be a connected reductive group defined over a local non-archimedean field $F$ with ring of integers $\cO$ and uniformizer $\pi.$ We denote by $G:=\bG(F)$ the group of $F$-points, and follow the same convention for parabolic subgroups, Levi subgroup etc.
	
	\subsection{Hecke algebras of $p$-adic reductive groups}
	
	The \textit{Hecke algebra of $G$} is the non-unital algebra $\cH(G)$ of locally constant compactly supported distributions on $G$ under convolution. For a smooth representation $\pi\in \mm(G)$ with underlying vector space $V_{\pi}$ and a fixed vector $v\in V_{\pi}$ we define a function $f_v: G\rightarrow V_{\pi}$ by $f_v(g)=\pi(g)v.$ We define a functor $F:\mm(G)\rightarrow \mm(\cH(G))$ sending a smooth representation $\pi\in \mm(G)$ with vector space $V$ to the $\cH(G)$-module $F(\pi):=V$ with action given by $F(\cE)v=\langle \cE, f_v \rangle.$ A module $M$ over a non-unital algebra $R$ is called non-degenerate if $Ann(x)\neq R$ for all $x\in M$. The next two positions are from \cite[\S 1]{Bernstein1992}.
	
	\begin{proposition}
		The functor $F$ defines an equivalence of categories
		$$\mm(G)\cong \mm(\cH(G)),$$
		where $\mm(\cH(G))$ denotes the category of non-degenerate $\cH(G)$-modules.
	\end{proposition}
	
	Let $K\leq G$ be a compact open subgroup, $\mm_K(G)$ the full subcategory of modules generated by $K$-fixed vectors, and $e_{K}$ the normalized constant distribution on $K$.
	
	\begin{definition}\label{comp}
		We define $\cH_K(G):=e_K\ast \cH \ast e_K \subseteq \cH(G)$ to be the subalgebra of $K$-biinvariant distributions.
	\end{definition}
	
	We have $\cH(G)=\bigcup_K \cH_K(G).$
	
	\begin{proposition}
		The functor $F$ restricts to an equivalence of categories
		$$\mm_K(G)\cong \mm(\cH_K(G)).$$
	\end{proposition}
	
	Another way to study $\mm(G)$ involves attaching Hecke algebras to the Bernstein blocks $\mm_{\fs}$ coming from Bernstein's decomposition theorem. 
	
	\begin{definition}
		A projective generator in an abelian category $\mm$ is a projective object $\Pi$ such that the functor $F_{\Pi}: \mm\rightarrow \Sets$ defined by $F_{\Pi}(X)=\Hom(\Pi,X)$ is faithful and preserves direct sums.
	\end{definition}	
	
	The following is \cite[Theorem 1.3]{Hyman1968}.
	
	\begin{lemma}\label{progenerator}
		Let $\mm$ be an abelian category with arbitrary direct sums that has a finitely generated projective generator $\Pi$. Let $\Lambda=\ed_{\mm}(\Pi)$ be the algebra of endomorphisms of $\Pi$ in $\mm$. Then 
		$$\mm\cong \mm(\Lambda^{\mathrm{op}})$$
		the category of right $\Lambda$-modules.
	\end{lemma}	
	
	\begin{definition}\label{amor}
		We call two algebras $\cH_1, \cH_2$ Morita equivalent if
		$$\mm(\cH_1)\cong \mm(\cH_2).$$
		We call $\cH_1, \cH_2$ almost Morita equivalent if 
		$$\mm^{\mathrm{f}}(\cH_1)\cong \mm^{\mathrm{f}}(\cH_2).$$
		where $\mm^{\mathrm{f}}(\cH)$ is the category of finite-dimensional modules. We denote Morita equivalence by $\cH_1\sim \cH_2$ and almost Morita equivalence by $\cH_1\amo \cH_2$.
	\end{definition}
	
	\begin{remark}
		A finitely generated projective generator is not unique and different algebras $\cH_1, \cH_2$ produced by Lemma \ref{progenerator} do not need to be isomorphic. Instead, since $\mm(\cH_1)\cong \mm \cong \mm(\cH_2),$ we always have the Morita equivalence
		$$\cH_1\sim \cH_2.$$
	\end{remark}

	Lemma \ref{progenerator} can be applied $\mm_{\fs}(G)$ due to the following proposition.
	
	\begin{proposition}
		$\mm_{\fs}(G)$ has a finitely generated projective generator $\Pi_G^{\fs}$. In particular, if $\cH_{\fs}:=\ed^{\mathrm{op}}_{\mm}(\Pi_G^{\fs}),$ we have
		$$\mm_{\fs}(G)\cong \mm(\cH_{\fs}(G)).$$
	\end{proposition}
	
	\begin{proof}
		The first assertion is \cite[Theorem 23]{Bernstein1992}, see also \cite[\S 1.6]{Roche2002}. The second follows by Lemma \ref{progenerator}.
	\end{proof}
	
	\begin{remark}
		For $\mm_{\fs}(G)$ there are actually many finitely generated projective generators, see \cite[\S 1]{Solleveld2022}. 
	\end{remark}
	
	\begin{remark}\label{hcenter}
		The center of $\cH_{\fs}$ which is canonical, can be identified with regular functions on the corresponding component of the Bernstein variety $\om_{\fs}(G)$ which was discussed in the introduction.
	\end{remark}

	\subsection{Iwahori-Hecke algebras}
	
	Let $G$ be a split reductive group over $F$. The algebras $\cH_K(G)$ and $\cH_{\fs}(G)$ can be very different. The exception is the unramified Bernstein component $\fs_0\in \fB(G)$ containing the representations generated by Iwahori-fixed vectors, ie, if $I$ is an Iwahori subgroup of $G$, we have $\cH_{\fs_0}(G)\cong \cH_I(G).$ We call $H:=\cH_I(G)$ the Iwahori-Hecke algebra of $G$. As it will be the main example of our methods in Section 5, we recall its structure theory. We recommend the exposition \cite{Haines2010}. Let us recall some of the structure theory discussed in \cite[\S 1]{Haines2010}.
	
	Let $T$ be a split maximal torus in $G$ with dual torus $\hat{T}$ over $\mathbb{C}$ and define $R=\C[X_*(T)]\cong \cO(\hat{T})$ to be the group algebra of the cocharacter lattice. Notice that by mapping $\mu\in X_*(T)$ to $\pi^{\mu}=\mu(\pi)\in T(F)$ we get an isomorphism $X_*(T)\cong T(F)/T(\oo)$. Using this isomorphism, we view $R$ as a representation of $T$ and after trivially extending to the Borel $B=TN$, we can consider the parabolic induction $i_B^G(R)$. Then the module of $I$-fixed vectors $M:=i_B^G(R)^I$ is a finitely generated projective generator for $\mm(H)$, so $H\cong \ed_H(M).$  By definition, $M\cong C_c(T(\oo)N\backslash G/ I)$. We define the extended affine Weyl group $\widetilde{W}\cong T(\oo)N\backslash G/I$ and by setting $v_x=1_{T(\oo)NxI}$ we get a $\widetilde{W}$-basis of $M$ as a vector space. Thus, we can define a left $R$-action on $M$ by $\pi^{\mu}\cdot v_x=q^{-\langle \rho,\mu\rangle}v_{\pi^{\mu}\cdot x}$, where $\rho$ is the half-sum of roots of $T$ in $\mathrm{Lie}(N)$, and therefore we get an embedding $R\hookrightarrow H$. Let $W$ be the finite Weyl group.
	
	$H$ is generated over $R$ by the elements $T_w:=1_{IwI}$, $w\in W$. We have $T_w=T_{s_1}\cdots T_{s_n}$ for any reduced expression $w=s_1\cdots s_n.$ To derive a complete presentation of $H$, it is enough to detect how the elements $T_{s}$ for a simple reflection $s$ interact with $R$ and an expression for $T_s^2$. For the latter, a simple integral calculation gives $T_s^2=(q-1)T_s+q$ where $q$ is the cardinality of the residue field. For the first, one way to proceed (see \cite[\S 1.10]{Haines2010}) is by defining \textit{intertwining operators} $I_w, \ \forall w\in W$ between suitable completions of $M$ satisfying the semilinear relation $I_w\pi^{\mu}=\pi^{w(\mu)}I_w.$ Alternatively, we can identify $I_w$ with elements $i_w$ of the algebra $H_L=H\otimes_R L$ where $L$ is the fraction field of $R$. By considering the actions of $I_w, T_w$ on the spherical vector $v_1$, and letting $a\in \Delta$ where $\Delta$ is a choice of simple roots, it is possible to determine explicitly \cite[\S 1.14]{Haines2010} that 	$$i_a=q^{-1}T_{s_a}+\frac{(1-q^{-1})\pi^{a^{\vee}}}{1-\pi^{a^{\vee}}}.$$	
	
	For a simple root $a\in \Delta$ let $s_a$ be the corresponding reflection, $a^{\vee}$ the dual coroot, and $i_a=i_{s_a}$. To ease computations, we define the elements $e_a=1-q^{-1}\pi^{a^{\vee}}$ and $d_a=1-\pi^{a^{\vee}}$, as well as $$c_a=\frac{e_a}{d_a}=\frac{1-q^{-1}\pi^{a^{\vee}}}{1-\pi^{a^{\vee}}}\in L.$$
		
	Then the quadratic relation $T_s^2=(q-1)T_s+q$ becomes
	\begin{equation}\label{iquad}
		i_a^2=c_as_a(c_a)
	\end{equation}
	
	By induction on the length of $w$, we get the following Lemma that will be useful in Section 5.
	
	\begin{lemma}\label{int}
		Let $w\in W$ and $R_w=\{ a\in \Delta \mid a>0, w(a)<0\}.$ Then,
		$$I_wI_{w^{-1}}=\prod_{a\in R_w}\frac{e_ae_{-a}}{d_ad_{-a}}.$$
	\end{lemma}
	
	We can also derive the intertwining relation satisfied by the $T_{s_a}$ from the intertwining relation $i_a r=s_a(r)i_a$. Combining with the quadratic relation, we get the Bernstein presentation for the Iwahori-Hecke algebra \cite[\S 1.15]{Haines2010}.

	\begin{proposition}\label{beriha}
		The Iwahori-Hecke algebra $H$ of $G$ is a free algebra over the group algebra of the cocharacter lattice $R$, with a basis given by elements $T_{w}=T_{s_1}\cdots T_{s_n}$ where $w=s_1\cdots s_n$ is a reduced expression for $w$.
		
		Any reduced expression provides the same $T_w$ and for a simple root $a$ the element $T_{s_a}$ satisfies
		\begin{eqnarray}
			(T_{s_a}+1)(T_{s_a}-q)=0 \\
			T_{s_a}\pi^{\mu}= \pi^{s_a(\mu)} T_{s_a} + \frac{(q-1)(\pi^{\mu}-\pi^{s_a(\mu)} )}{1-\pi^{-a^{\vee}}}.
		\end{eqnarray}
		
	\end{proposition}
	
	We can also use the intertwining elements, to determine the center of $H$. From the intertwining relations, it follows that $W$-invariant elements of $R$ are in the center, ie. $R^W\subseteq Z(H)$. The other direction is also true by virtue of the Satake isomorphism.
	
	\begin{theorem}[Satake isomorphism]
		The center $Z(H)$ of $H$ is $R^W\cong \cO(\hat{T}\git W)$.
	\end{theorem}

	\subsection{Affine Hecke algebras}
	
	The results of the previous subsection admit generalizations in the sense of providing specific presentations for algebras that are (almost) Morita equivalent to Hecke algebras of Bernstein components. Affine Hecke algebras have since found applications in various areas of mathematics, such as knot theory, combinatorics, etc. Due to the fact they admit a Bernstein presentation, Theorem \ref{main} can be applied. We recommend the exposition \cite{Solleveld2021}. 
	
	\begin{proposition}
		Let $(W,S)$ be a Coxeter group, equipped with a function $q: S\rightarrow \C$ such that $q(s)=q(s')$ if $s,s'$ are conjugate in $W$. There is a unique algebra structure $\cH(W,q)$ on the vector space over $\C$ generated by elements $T_w, w\in W$ such that 
		\begin{itemize}
			\item $T_e=1$,
			\item $(T_s-q(s))(T_s+1)=0, s\in S$,
			\item $T_{w_1w_2}=T_{w_1}T_{w_2}$ if $l(w_1w_2)=l(w_1)+l(w_2)$.
		\end{itemize}
		We call $\cH(W,q)$ the Iwahori-Hecke algebra of $(W,q)$. If $W$ is finite, we call $\cH(W,q)$ its finite Hecke algebra. If $(W,S)$ is an affine Weyl group, we say $\cH(W,q)$ is of affine type.
	\end{proposition}
	
	The connection with the definitions of the previous subsection is the following theorem \cite{Iwahori1965}. Let $\cR(G,T)$ denote the root datum of $G$ with respect to $T$. Recall that a root datum $\cR=(X^*, \Phi, X_*, \Phi^{\vee})$ comes equipped with the finite Weyl group $W=W(\Phi)$ of its root system, an affine Weyl group $W_{\mathrm{aff}}$, and the extended affine Weyl group $W(\cR)=X^*\rtimes W(\Phi)$, the first two of which are Coxeter groups.
	
	\begin{proposition}\label{impr}[Iwahori-Matsumoto presentation]
		If $G$ is a split, simply connected, semisimple group over a non-archimedean local field $F$, $T$ a split maximal torus, and $(W_{\mathrm{aff}},S)$ the affine Weyl group of the dual root datum $\cR(G,T)^{\vee}$. Let $q(s)=p, \ \forall s\in S.$ Then,
		$$\cH_{I}(G)\cong \cH(W_{\mathrm{aff}},q).$$
	\end{proposition} 
	
	Affine Hecke algebras are a generalization of Iwahori-Hecke algebras of affine type.
	
	\begin{proposition}
		Let $\cR=(X^*, \Phi, X_*, \Phi^{\vee})$ be an irreducible root datum with finite Weyl group $W=W(\Phi)$, $q\in \R_{\geq 1}$, and $\lambda, \lambda^*: \Phi\rightarrow \C$ be $W$-invariant functions such that
		
		$$a^{\vee}\notin 2X_* \implies \lambda(a)=\lambda^*(a).$$
		
		Let $\C[X^*]$ be the group algebra of the character lattice, with the standard basis $\{\theta_x, x\in X^*\}$ and $\cH(W,q)$ the finite Hecke algebra of $W$.
		
		There is a unique algebra structure on the vector space $\cH(\cR,\lambda,\lambda^*,q):=\C[X^*]\otimes \cH(W,q)$ such that the following are true.
		
		\begin{itemize}
			\item $\C[X^*],\cH(W,q)$ are embedded as subalgebras.
			\item For $a\in \Delta, x\in X$
			$$\theta_xT_{s_a}-T_{s_a}\theta_{s_a(x)}=\left((q^{\lambda(a)}-1)+\theta_{-a}\left(q^{(\lambda(a)+\lambda^*(a))/2}-q^{(\lambda(a)-\lambda^*(a))/2}\right)\right)\frac{\theta_x-\theta_{s_a(x)}}{\theta_0-\theta_{-2a}}.$$
		\end{itemize}
		We call $\cH(\cR,\lambda,\lambda^*,q)$ the affine Hecke algebra of $\cR$.
		
		If $\lambda(a)=\lambda(b)=\lambda^*(a)=\lambda^*(b)$ for all $a,b$, we say $\cH(\cR,\lambda,\lambda^*,q)$ has equal parameters.  
		
	\end{proposition}

	\begin{remark}
		Notice that if $a^{\vee}\notin 2X_*$ or if $\cH(\cR,\lambda,\lambda^*,q)$ has equal parameters, the second relation simplifies to
		$$\theta_xT_{s_a}-T_{s_a}\theta_{s_a(x)}=(q^{\lambda(a)}-1)\frac{\theta_x-\theta_{s_a(x)}}{\theta_0-\theta_{-a}}.$$
	\end{remark}
	
	\begin{remark}
		If $\cR$ is not irreducible, let $d$ be the number of connected components. As noticed in \cite{Aubert2021} , the proposition remains true if we substitute $\bz=(\bz_1,\ldots,\bz_d)$ in the place of $q$ with the obvious changes in the relations. We also call this an affine Hecke algebra $\cH(\cR,\lambda,\lambda^*,\bz).$ 		
	\end{remark}	
	
	Not every affine Hecke algebra is an Iwahori-Hecke algebra, but we do have a generalization of the Iwahori-Matsumoto presentation. Let $\Omega:=\{w\in W(\cR)\mid l(w)=0\}$ be the elements of length zero. Then, $$W(\cR)=W_{\mathrm{aff}}\rtimes \Omega.$$
	
	\begin{proposition}
		There is a unique algebra isomorphism $$\cH(\cR,\lambda,\lambda^*,q)\xrightarrow[]{\cong} \cH(W_{\mathrm{aff}},q)\rtimes \Omega$$ such that 
		\begin{itemize}
			\item It restricts to the identity in $\cH(W,q).$
			\item For all $x\in \Z \Phi$ with $\langle x, a^{\vee}\rangle\geq 0, \ \forall \ a\in \Delta$, it sends $\theta_x$ to $q(x)^{-1/2}T_x.$
		\end{itemize}
	\end{proposition}
	
	For many Bernstein components $\fs\in \fB(G)$, $\cH_{\fs}$ is Morita equivalent to an affine Hecke algebra. In \cite{Solleveld2022}, it was proven that one slight generalization is still needed.	
	
	\begin{definition}
		Consider the following data:
		\begin{enumerate}
			\item A root datum $\cR=(X^*, \Phi, X_*, \Phi^{\vee})$ and a choice of a set of simple roots $\Delta$,
			\item A finite group of the form $W=W(\Phi)\rtimes \fR$ acting on $X^*$,
			\item A $2$-cocycle $\natural: (W/W(\Phi))^2\rightarrow \C.$
			\item $W$-invariant functions $\lambda , \lambda^*: \Phi\rightarrow \C,$ such that $a^{\vee}\notin 2X_* \implies \lambda(a)=\lambda^*(a).$
			\item An array of invertible parameters $\bz=(\mathbf{z}_1,\ldots, \mathbf{z}_d)$.
		\end{enumerate}
		We define the twisted affine Hecke algebra to be $\cH(\cR,\lambda,\lambda^*,\bz)\rtimes \C[\fR,\natural].$
	\end{definition}
	
	For a specialization of the parameters $\bz=(z_1,\ldots,z_d)$, we can specialize a twisted Hecke algebra. 
	
	We also recall a more geometric construction that ties nicely with our perspective. Let $T:=\Hom(X^*, \C^{\times})$ a complex algebraic torus. Since $\cO(T)\cong \C[X^*]$, the group $W$ acts naturally on $T$. 
	
	\begin{proposition}
		There is a unique algebra structure on the vector space $$\cH(T,W,\lambda,\lambda^*, \natural,\bz):=\cO(T)\otimes \C[\bz,\bz^{-1}]\otimes \C[W(\Phi)]\otimes \C[\fR,\natural]$$
		such that 
		\begin{itemize}
			\item Under the isomorphism $\cO(T)\cong \C[X^*],$ the span of
			$\cO(T), \C[\bz,\bz^{-1}]$ and $\C[W(\Phi)]$ is the affine Hecke algebra $\cH(\cR,\lambda,\lambda^*,\bz).$
			\item $\C[\fR,\natural]$ embeds as a subalgebra.
			\item For $\gamma \in \fR$, $w\in W(R)$ and $x\in \cO(T)$:
			$$T_{\gamma}T_{w}\theta_x T_{\gamma}^{-1}=T_{\gamma x\gamma^{-1}}\theta_{\gamma(x)}.$$
		\end{itemize}
		If $\fR=1$, then $\cH(T,W,\lambda,\lambda^*, \natural,\bz)$ is the affine Hecke algebra of $\cR$.
	\end{proposition}
	
	\begin{proof}
		Similar to \cite[Proposition 2.2]{Aubert2021}.
	\end{proof}
	
	Our motivation for introducing twisted affine Hecke algebras is the next theorem which is the main theorem of \cite{Solleveld2022}. 
	
	\begin{theorem}\label{solleveld}
		Let $G$ be a reductive group and $\fs\in \fB(G)$. Then, there exist parameters $\lambda, \lambda^*$ and a cocycle $\natural$ such that $\cH_{\fs}$ is almost Morita equivalent to a specialization $H_{\fs}$ of the twisted affine Hecke algebra $\cH(T_{\fs},W_{\fs},\lambda,\lambda^*, \natural,\bz)$.
	\end{theorem}
	
	The classical Satake isomorphism admits the following generalization.
	
	\begin{lemma}\label{center}
		\sloppy $\cO(T\times \C^d)^{W}$ is a central subalgebra of $\cH(T,W,\lambda,\lambda^*, \natural,\bz)$. It equals $Z(\cH(T,W,\lambda,\lambda^*, \natural,\bz))$ if $W$ acts faithfully on $T$. For a specialization $H$, we have 
		$$Z(H)\cong \cO(T)^{W}.$$
	\end{lemma}
	
	\begin{proof}
		This is \cite[Lemma 2.3]{Aubert2021}.
	\end{proof}
	
	\begin{remark}
		Let $\bv_{\fs}(G)$ be the connected component of the Bernstein variety corresponding to $\fs\in \fB(G)$. It is actually the GIT quotient $T_{\fs}\git W_{\fs}$ of the torus and the Weyl group used to construct the algebras in Solleveld's theorem.
		
		Since almost Morita equivalence preserves the center, by Lemma \ref{center} we retrieve the classical fact that
		$$\oo(\bv_{\fs}(G))\cong Z(H_{\fs})\cong Z(\cH_{\fs}(G)),$$
		which was mentioned in Remark \ref{hcenter}.
	\end{remark}
	
	\section{The Lafforgue variety}
	
	The main goal of this Section is to give a proof of Theorem \ref{main}. Let $k$ be any field and fix an algebraic closure $\bar{k}$. Let $R$ be an associative unital $A$-algebra over a finitely generated commutative central $k$-subalgebra $A\subseteq Z(R)$ such that $R$ is finite as an $A$-module. Lafforgue's original assertion concerns the case of the Hecke algebra $\cH_K$ of a reductive $p$-adic group $G$ and a compact open subgroup $K\leq G$. In particular, Lafforgue's assertion is Theorem \ref{main} for $R=\cH_K$, $A=Z_K$ and $k=\C$. 
	
	If $k=\bar{k}$ was algebraically closed and $M$ is a finite dimensional simple $R$-module, then $A$ must act on $M$ through a character $a:A\to \bar{k}$ by Schur's lemma. The Lafforgue variety can thus be thought of as a scheme over $\Spec(A)$. 
	
	We construct a non-commutative Hilbert scheme and a nested non-commutative Hilbert scheme for the finite $A$-algebra $R$, without any assumptions on $k$, based on Grothendieck's classical results on Quot schemes that are recalled in the Appendix. 
	
	Assume now $\cha(k)=0.$ Then, we construct a trace map from the Hilbert scheme to a generalized Grothendieck vector bundle \cite[Notation 4.6.1.3]{Artin1963}. The Lafforgue variety will be defined to be the image of that map. 
	
	We use the same strategy in positive characteristic, albeit with a twist. Due to the elementary fact that a simple module is not determined by its traces in positive characteristic, we construct a determinant map based on Roby's concept of a polynomial law \cite[Chapitre I, \S 2]{Roby1963}.
	
	\subsection{Non-commutative Hilbert scheme}
	
	Let $R$ be a unital associative $k$-algebra that is a finite module over a finitely generated central $k$-subalgebra $A\subseteq Z:=Z(H)$.
	
	\begin{definition}
		We call the non-commutative Hilbert functor $\Hilb_{R/A}$ the functor that associates to every commutative $A$-algebra $B$ the set of isomorphism classes of $R\otimes_A B$-modules $M$, which are flat as $B$-modules, equipped with a surjective $R\otimes_A B$-linear map $R\otimes_A B\sur M$. 
	\end{definition}
	
	\begin{proposition}\label{hilb}
		The functor $\Hilb_{R/A}$ is representable by a proper scheme over $\Spec(A).$
	\end{proposition}
	
	\begin{proof}
		We consider $R$ just as a finite $A$-module. The Quot functor $\cQ_{R/A}$ associating to every commutative $A$-algebra $B$ the set of isomorphism classes of flat $B$-modules $M$ equipped with a surjective $R\otimes_A B$-linear map $m:R\otimes_A B\twoheadrightarrow M$ is representable by a projective scheme over $\Spec(A)$, by Theorem \ref{quotf} for $E_R$ the coherent sheaf on $S=\Spec(A)$ defined by $R$, and $T=\Spec(B)$. Since the functor $\Hilb_{R/A}$ is a closed subfunctor of $\cQ_{R/A}$, as the extra condition of being an $R\otimes_A B$-module is given by infinitely many algebraic equations, it is also representable by a projective scheme over $\Spec(A)$.
	\end{proof}	
	
	\begin{proposition}\label{hilbd}
		There is a decomposition of $\Hilb_{R/A}$ into open and closed subschemes
		\begin{equation}
			\Hilb_{R/A}=\bigsqcup_{d\in \N} \Hilb^d_{R/A}
		\end{equation}
		where $\Hilb^d_{R/A}$ classifies $R\otimes_A B$-linear maps $m:R\otimes_A B\sur M$ with $M$ being a locally free $B$-module of rank $d.$ 
	\end{proposition}
	
	\begin{proof} 
		The fact that $M$ is flat over $B$ and finitely presented due to the existence of the surjective map, implies that it is locally free over $B$ by \cite[\href{https://stacks.math.columbia.edu/tag/05P2}{Lemma 05P2}]{stacks-project} for $X=\Spec B$. The decomposition follows from the same decomposition for the Quot scheme, see Definition \ref{quotfull}.
	\end{proof}
	
	Since we are mainly interested in simple modules, it will also be useful to consider the following functor.
	
	\begin{definition}
		The nested non-commutative Hilbert functor $n\Hilb_{R/A}$ is the functor which associates to every commutative $A$-algebra $B$ the set of isomorphism classes of pairs of $R\otimes_A B$-modules $M, N$, which are flat as $B$-modules, equipped with surjective $R\otimes_A B$-linear maps $R\otimes_A B\twoheadrightarrow M \twoheadrightarrow N$, where we also require that the latter map has a non-zero kernel.
	\end{definition}
	
	\begin{proposition}
		The functor $n\Hilb_{R/A}$ is representable by a proper scheme over $\Spec(A)$. 
	\end{proposition}
	
	\begin{proof}
		The proof is the same as of Proposition \ref{hilb} due to the fact that $n\Hilb_{R/A}$ is similarly a closed subscheme of a nested Quot scheme as in Theorem \ref{nquot}.
	\end{proof}

	\begin{proposition}\label{nhilb}
		The forgetful map 
		$$F_N: n\Hilb_{R/A} \to \Hilb_{R/A}$$
		defined by $F_N(M,N):=N$ is a proper morphism. In particular, the complement $i\Hilb_{R/A}$ of the image of $F_N$ is open. 
		
		The $B$-points of $i\Hilb_{R/A}$ correspond to simple modules.
	\end{proposition}
	
	\begin{proof}
		By the fact that $n\Hilb_{R/A}$ is proper over $\Spec(A)$ and $\Hilb_{R/A}$ is separated and locally of finite type over $\Spec(A)$, the first part follows from \cite[\href{https://stacks.math.columbia.edu/tag/01W6}{Lemma 01W6}]{stacks-project}. The second part readily follows by the observation that if $M$ is not simple, it must admit some proper quotient $M\twoheadrightarrow N$ and thus be in the image of $n\Hilb_{R/A}$.
	\end{proof}
	A geometric point $x\in i\Hilb_{R/A}(\bar k)$ over a point $a:A\to \bar k$ consists of a quotient $M_x$ of the algebra $R_a=R\otimes_A \bar k$ by a maximal left ideal, or in other words, $M_x$ is a simple $R_a$-module equipped with a generator.

	We consider the group functor $\cG_{R/A}$ which associates to every commutative $A$-algebra $B$ the group $(R\otimes_A B)^\times$ of invertible elements of the possibly non-commutative algebra $R\otimes_A B$.

	\begin{proposition}\label{groupsch}
		If $R$ is locally free over $A$, $\cG_{R/A}$ is a group scheme.
	\end{proposition}
	
	\begin{proof}
		This follows analogously to the easy direction of \cite[Theorem 1]{Nitsure2002}. Let $W(R)(B):=R\otimes_A B$ be the functor denoted by the same letter of \cite[Notation 4.6.1]{Artin1963}. Then in the locally free case $W(R)=V(R^{\vee})$ where notation is as in \emph{loc. cit.}, and the functor $V$ is representable by \cite[Notation 4.6.3.1]{Artin1963}.
		
		Then $\cG_{R/A}$ is identified as a closed subscheme of $W(R)\times W(R)$ by requiring that the product of two elements is $1$. The group structure is obvious.
	\end{proof}
	
	\begin{remark}
		In the case $R$ is locally free over $A$, we also have that the non-commutative Hilbert scheme is a closed subscheme of the relative Grassmannian, see Theorem \ref{quotf}.
	\end{remark}

	The group functor $\cG_{R/A}$ acts on $\Hilb_{R/A}$ relative to $\Spec(A)$. For a $B$-point $(M,m)\in\Hilb_{R/A}(B)$ we will denote the action of $R\otimes_A B$ on $M$ by $(r,m)\mapsto e(r)m$. If $g\in (R\otimes_A B)^\times$ we define the action of $g$ on $(M,m)$ to be $g(M,m)=(M',m')$ where $M'=M$ as a $B$-module equipped with the structure of $R\otimes_A B$-module given by $e'(r)m= e(g^{-1}rg)m$, and $m'=e(g)m$. Similarly, the group functor $\cG_{R/A}$ also acts on the nested Hilbert scheme $n\Hilb_{R/A}.$ 
	
	\begin{proposition}
		The morphism $F_N: n\Hilb_{R/A} \to \Hilb_{R/A}$ is $\cG_{R/A}$-equivariant, and the complement of its image $i\Hilb_{R/A}$ is open and stable under the action of $\cG_{R/A}$.
	\end{proposition}
	\begin{proof}
		The first part follows from the definition of the action. Therefore, the image is a $\cG_{R/A}$-equivariant closed subscheme of $\Hilb_{R/A}$. The second part follows from this observation and Proposition \ref{nhilb}.
	\end{proof}
	
	\subsection{Trace map}
	
	We define a generalized vector bundle $V_{R/A}$ over a commutative ring attached to an $A$-module such as $R$. As a functor, $V_{R/A}$ attaches to each $A$-algebra $B$ the abelian group $\Hom_A(R,B)$ of all $A$-linear maps $R\to B$. This functor is represented by the symmetric algebra $\Sym_A(R)$: it is the $\N$-graded $A$-algebra with $\Sym^0_A(R)=A$, $\Sym^1_A(R)=R$, and for every $d\in \N$, the $d$-th symmetric power $\Sym^d_A(R)$ is the largest quotient of the $d$th fold tensor power $R^{\otimes d}$ of $R$ over $A$ on which the symmetric group $S_d$ acts trivially. We claim that the morphism of functors on $A$-algebras:
	$$\Hom_{A-Alg}(\Sym_A(R),B) \to \Hom_A(R,B),$$
	defined as the restriction of an $A$-algebra homomorphism $x:\Sym_A(R)\to B$ to the degree 1 component $\Sym^1_A(R)=R$, is an isomorphism of functors. Indeed, every $A$-linear map $y:R\to B$, induces an $A$-linear map $R^{\otimes d} \to B^{\otimes d} \to B$ which factors through an $A$-linear map $y^d :\Sym^d_A(R)\to B$. It's not hard to check that the $A$-linear map $x: \bigoplus_{d\in\N} \Sym^d_A(R) \to B$ given by $x=\bigoplus_{d\in\N} y^d$ is a homomorphism of $A$-algebras. It is also clear that the map $y\mapsto x$ thus defined gives rise to an inverse of the functor $x\mapsto y$. We conclude that the functor $V_{R/A}$ is representable by the affine scheme $\Spec(\Sym_A(R))$ which is a generalized vector bundle in the sense of Grothendieck.
	
	We can construct the trace map 
	\begin{equation}
		\tr_{R/A}: \Hilb_{R/A} \to V_{R/A}
	\end{equation}
	as follows. By Proposition \ref{hilbd}, every point of $\Hilb_{R/A}(B)$ can be written as $(M,m)$ where $M$ is an $R\otimes_A B$-module that is locally free and finite as a $B$-module and $m$ is the image of $1$. Every $r\in R$ defines a  $B$-linear operator of $M$ given by the structure of an $R\otimes_A B$-module. Since $M$ is a finitely generated locally free $B$-module, the trace $\tr_B(r)\in B$ is well defined. This gives rise to an $A$-linear map $\tr_M:R\to B$ and thus to a $B$-point of $V_{R/A}$. 
	
	We now define $\Laf_{R/A}$ to be the scheme-theoretic image of $\tr_{R/A}$. Let $\iLaf_{R/A}$ denote the complement of the image of the composition $\tr_{R/A} \circ F_N: n\Hilb_{R/A} \to \Laf_{R/A}$, where said image is closed because $n\Hilb_{R/A}$ is proper. 
	
	For a point $a:A\rightarrow k(a)$, we denote by $R_a$ the finite dimensional $k(a)$-algebra $R_{a}=R\otimes_{A,a}k(a).$
	
	\begin{proposition}\label{preimage}
		Assume that $A$ is a $k$-algebra where $k$ is a field of characteristic zero. Then the preimage $\tr_{R/A}^{-1}(\iLaf_{R/A})$ is $i\Hilb_{R/A}$. Moreover for every geometric point $l\in \iLaf_{R/A}(\bar k)$ over $a:A\to \bar k$, the group $\cG_{R/A}(\bar k)$ acts transitively on the fiber $\tr_{R/A}^{-1}(l)$. 
	\end{proposition} 
	
	This assertion is nothing but a reformulation of well known facts about modules over a finite-dimensional algebra, improperly referred to as Brauer-Nesbitt's theorem \cite[chapter XVII Cor 3.8]{Lang2002}. As we want to extend the construction of the Lafforgue variety to the case of positive characteristic, we will give a sketch of the proof of Proposition \ref{fibers} which was given in full details in Lang's book to which we refer for more information. 
	
	\begin{proposition}[Bourbaki] \label{fibers} Assume that  $k$ is a field of characteristic zero and $R$ is a finite-dimensional $k$-algebra possibly non-commutative. Let $M$ and $N$ be $k$-finite dimensional $R$-modules such that for all $x\in R$, we have $\tr_x(M)=\tr_x(N)$, then $M$ and $N$ have the same semi-simplification. In particular if $\tr_R(M_1)=\tr_R(M_2)$, and if $M_1$ is a simple $R$-module then $M_2$ is also simple and $M_2\cong M_1$. 
	\end{proposition}
	
	\begin{proof}
		The assertion is obvious in one direction. If the $R$-modules $M_1$ and $M_2$ have the same semi-simplification as $R$-modules then the induced linear forms $\tr_{M_1},\tr_{M_2}: R\to k$ are equal because traces only depend on the semi-simplification. Conversely, Jacobson's density theorem \cite[Chapter XVII, Theorem 3.2]{Lang2002} (the original paper is \cite{Jacobson1945}) implies the existence of projectors: if $V_0,V_1,\ldots,V_n$ are non-isomorphic simple $R$-modules, there exists an element $e_0\in R$ which acts as the identity on $V_0$ and $0$ on $V_1,\ldots,V_n$. Now let $V_0,\ldots,V_n$ be the simple $R$-modules occurring as a simple subquotient of $M_1$ or $M_2$ and write decompositions of the semi-simplifications of $M_1$ and $M_2$ as $M_1^{ss}= V_0^{m_1} \oplus U_1$ and $M_2^{ss}= V_0^{m_2} \oplus U_2$ where $U_1$ and $U_2$ are semi-simple modules with no occurrences of $V_0$. We have the equalities 
		$$m_1\dim(V_0)=\tr_{M_1}(e_0)= \tr_{M_2}(e_0)=m_2\dim(V_0)$$
		as elements of $k$, and since $\cha(k)=0$ we have $m_1=m_2$. The same argument applies in any $V_i$, therefore $M_1$ and $M_2$ have the same semi-simplification.
	\end{proof}
	
	\begin{remark}
		Jacobson's density theorem is valid in any characteristic. The characteristic zero assumption is only used to guarantee that in $k$ we have $\dim(V_0)\neq 0$. 
	\end{remark}
	
	We turn to the proof of Proposition \ref{preimage}.
	
	\begin{proof}
		The last statement of Proposition \ref{fibers} implies that $\tr_L^{-1}(\iLaf_{R/A})=i\Hilb_{R/A}$. It also implies that if $l\in \iLaf_{R/A}(\bar k)$ over $a:A\to \bar k$, and if $x_1,x_2\in \tr_L^{-1}(l)$ are represented by $(M_1,m_1)$ and $(M_2,m_2$ of $R_a$ then $M_1$ and $M_2$ are isomorphic simple $R_a$-modules. It follows that there exists $g\in R_a^\times$ such that $gx_1=x_2$. In other words, the fiber of $\cG_{R/A}$ over $a$ acts transitively on the fiber of $\tr$ over $l$. 	
	\end{proof}
	
	Finally, we recall and prove Theorem \ref{main}, in the case $\cha(k)=0$.
	
	\main*
	
	\begin{proof}
		By Proposition \ref{fibers}, the trace map $\tr: \Hilb_{R/A}\rightarrow \Laf_{R/A}$ forgets the choice of a generator for the module $M$ and parametrizes modules up to isomorphism. Since $\iLaf_{R/A}$ is the complement of the image $\tr \circ F_N$, modules in $\iLaf_{R/A}$ are the ones not admitting a proper quotient $M\twoheadrightarrow N$, therefore they are simple.
		
		Since $\Hilb_{R/A}$ is proper over $\Spec A$ and $V_{R/A}$ is separated and locally of finite type, $p$ is proper by \cite[\href{https://stacks.math.columbia.edu/tag/0AH6}{Tag 0AH6}]{stacks-project}. Since $\Laf_{R/A}$ is a closed subscheme of the affine scheme $V_{R/A}$, $p$ is affine. Therefore, $p$ is finite.
	\end{proof}

	\subsection{Determinant map}
	
	Without any hypothesis on the characteristic, we have to replace the trace by the determinant. Let us formalize the construction of the determinant map as an analogue of the trace map previously defined. 
	
	Let $A$ be a commutative ring. An $A$-module $R$ gives rise to a functor $\underline R: B\mapsto R\otimes_A B$ from the category of $A$-algebras to the category of abelian groups. We recall the definition of a polynomial law in \cite{Roby1963}.
	
	\begin{definition}
		A polynomial law on $R$ is a morphism of functors $f:\underline R \to \underline A.$
		
		We denote by $\mathrm{Pol}_A(R)$ the set of all polynomial laws on the $A$-module $R$.
	\end{definition} 
	Thus, a polynomial law $f$ on $R$ consists of a family of homomorphisms of abelian groups $f_B:R\otimes_A B\to B$ depending on $B$ in a functorial way. 
	
	If $r_1,\ldots,r_n\in R\otimes_A B$ form a finite sequence $\underline r$ of elements of $R\otimes_A B$, then $f$ gives rise to a polynomial $f_{\underline r}\in B[X_1,\ldots,X_n]$, where $X_1,\ldots,X_n$ are free variables, such that for every $x_1,\ldots,x_n \in B$, we have $f_{\underline r}(x_1,\ldots,x_n)=f(x_1 r_1+\cdots+x_n r_n)$. Indeed, if we take $$X_{\underline r}=r_1\otimes X_1+\cdots + r_n \otimes X_n \in R\otimes_A B[X_1,\ldots,X_n],$$ then we set 
	$$f_{\underline r}:=f_B(X_{\underline r})\in B[X_1,\ldots,X_n],$$ 
	see \cite[Thm 1.1]{Roby1963}. The main point in Roby's concept of polynomial law is that the polynomial $f_{\underline r}$ is a part of the data of $f$.
	
	\begin{definition}
		We say that the polynomial law $f:\underline R\to \underline A$ is homogeneous of degree $d\in\N$ if for every $A$-algebra $B$, $x\in B$ and $r\in R\otimes_A B$ we have $f_B(xr)=x^d f_B(r)$. 
		
		We denote by $\mathrm{Pol}^d_A(R)$ the set of all homogeneous polynomial laws of degree $d$ on the $A$-module $R$.
	\end{definition}	
	It is not hard to check that the polynomial law $f$ is homogeneous of degree $d$ if and only if for every finite sequence $\underline r=(r_1,\ldots,r_n)$ of elements of $R\otimes_A B$ for any $A$-algebra $B$, $f_{\underline r}$ is a homogeneous polynomial of degree $d$ with coefficients in $B$, \cite[Prop. I.1, p. 226]{Roby1963}.  
	
	\begin{example}
		A homogeneous polynomial law of degree $1$ on $R$ consists of a family of linear forms $f_B:M\otimes_A B\to B$ depending functorially on $B$.  
	\end{example}
	
	Now, we will generalize Grothendieck's construction of a generalized vector bundle associated to an $A$-module, by replacing linear forms on $M$ by homogeneous polynomial laws of degree $d$. Let $A$ be a commutative ring and $R$ a $A$-module. 
	
	\begin{definition}
		We define the functor $S^d V_{R/A}$ which attaches to every $A$-algebra $B$ the set $\mathrm{Pol}^d_B(R\otimes_A B)$ of polynomial laws on the $B$-module $R\otimes_A B$ which are homogeneous of degree $d$.
	\end{definition} 
	
	\begin{example}
		For $d=1$, $\mathrm{Pol}^1_B(R\otimes_A B)=\Hom_B(R,B)$ and we have an isomorphism of functors $S^1 V_{R/A}=V_{R/A}$ which is represented by the affine scheme $\Spec (\Sym_A(R))$.
	\end{example}
	
	The above example generalizes. Let $\Gamma^d_A R$ denote the $d$-th divided power of the $A$-module $M$ \cite[Ch. III, p. 249]{Roby1963}. The following proposition is \cite[Thm III.3 p. 262, IV.1 p. 266]{Roby1963}.

	\begin{proposition}
		For every $d\in\N$, there is a canonical isomorphism of functors $\mathrm{Pol}^d_B(R\otimes_A B)=\Hom_B(\Gamma^d_A R,B).$
	\end{proposition}

	As a consequence, we obtain the following.
	\begin{proposition}
		The functor $B\mapsto S^d V_{R/A}(B)=\mathrm{Pol}^d_B(R\otimes_A B)$ is representable by the affine scheme $\Spec(\Sym_A (\Gamma^d_A R))$. 
	\end{proposition}	
	
	Let $R$ be a possibly non-commutative algebra containing a commutative ring $A$ in its center such that $R$ is finite locally free as an $A$-module. 
	
	Every point $x\in \Hilb^d_{R/A}(B)$ is represented by an $(R\otimes_A B)$-quotient module $M$ of $R\otimes_A B$ which, as a $B$-module, is locally free of rank $d$. Choosing local generators of $M$ as a locally free $B$-module we see that $r\mapsto \det_M(r)$ gives rise to a morphism $R\otimes_A B \to B$ which is homogeneous of degree $d$ and therefore a point $\det(x)\in S^d V_{R/A}(B)$.
	
	\begin{definition}
		We call the morphism
		\begin{equation}
			\mathrm{det}_{R/A}: \Hilb^d_{R/A} \to S^d V_{R/A}
		\end{equation}
		defined by the morphism of functors $x\mapsto \det(x)$ the determinant map.
	\end{definition}
	
	As in the proof of Theorem \ref{main} in the previous subsection, since $\Hilb^d_{R/A}$ is a proper scheme over $A$, and $S^d V_{R/A}$ is affine, the morphism $\mathrm{det}_{R/A}$ factors through a closed subscheme $\Laf^d_{R/A}$ of $S^d V_{R/A}$ which is finite over $A$. We thus get a proper surjective map
	$$\mathrm{det}_L^d: \Hilb^d_{R/A} \to \Laf^d_{R/A}.$$	
	
	Using the nested Hilbert scheme $n\Hilb^d_{R/A}$ as before, we can define open subschemes $i\Hilb^d_{R/A}$ and $i\Laf^d_{R/A}$ as the complements of $n\Hilb^d_{R/A}$ and its image. Geometric points $x\in i\Hilb^d_{R/A}(\bar k)$ over $a:A\to \bar k$ correspond to $R_a$-modules that are simple. 
	
	\begin{proposition}
		We have $\det_L^{-1}(i\Laf^d_{R/A})=i\Hilb^d_{R/A}$. For every geometric point $l\in i\Hilb^d_{R/A}(\bar k)$ over $a:A\to \bar k$, the group functor $\cG_{R_a}$ acts transitively on the fiber $(\det_L^{d})^{-1}(l)$. 
	\end{proposition}
	
	Again, this assertion is nothing but a reformulation of well known facts about modules over a finite-dimensional algebra, which we prove for completeness.
	
	\begin{proposition}
		Let $R$ be a possibly non-commutative finite-dimensional algebra over a field $k$ and $M,N$ be $R$-modules which are $d$-dimensional $k$-vector spaces. Assume that $\det_{M}=\det_{N}$ as homogeneous polynomial of degree $d$ on $R$, then $M$ and $N$ have isomorphic semi-simplifications. In particular, if $M$ is a simple $R$-module then $N$ is also simple and $N\cong M$. 
	\end{proposition}
	
	\begin{proof}
		The assertion is obvious in one direction. If the $R$-modules $M$ and $N$ have the same semi-simplification then the induced homogenous forms $\det_{M},\det_{N}: R\to k$ are equal because the determinant only depends on the semi-simplification. Conversely, Jacobson's density theorem \cite[Chapter XVII, Theorem 3.2]{Lang2002} implies the existence of projectors: if $V_1,\ldots,V_r$ are non-isomorphic simple $R$-modules, then there exists an element $e_i\in R$ which acts as identity on $V_i$ and $0$ on $V_j$ for $j\neq i$. Now let $V_1,\ldots,V_r$ be the simple $R$-modules occurring as a simple subfactors of $M$ or $N$ and  decompose the semi-simplifications of $M$ and $N$ as 
		\begin{equation}
			M^{ss}= V_1^{m_1} \oplus\cdots\oplus V_r^{m_r} \mbox{ and }
			N^{ss}= V_1^{n_1} \oplus \cdots\oplus V_r^{n_r}
		\end{equation} 
		If $X_1,\ldots,X_r$ are free variables then we have the formula 
		$$\mathrm{det}_M\left(X_1e_1+\cdots+X_r e_r\right)= X_1^{m_1 \dim(V_1)} \ldots X_r^{m_r \dim(V_r)}$$ 
		for the determinant of $x_1e_1+\cdots+x_r e_r$ on $M$ and similarly for $N$.  The equality 
		$$\mathrm{det}_M\left(X_1e_1+\cdots+X_r e_r\right)=
		\mathrm{det}_N\left(X_1e_1+\cdots+X_r e_r\right)$$ 
		of polynomials of variables $X_1,\ldots,X_r$ implies that $m_i=n_i$ for all $i$. It follows that $M$ and $N$ have isomorphic semi-simplifications.
	\end{proof}
	The proof of Theorem \ref{main} in the general case readily follows exactly as in the previous subsection by replacing the trace map by the determinant map.
	
	\begin{remark}
		Notice that in the positive characteristic case the ring of regular functions $T_R$ on the Lafforgue variety is not given anymore via the simple procedure described in the introduction.
	\end{remark}
	
	\subsection{Dependence on the central subalgebra}
	
	If we assume $k$ to be algebraically closed, then recall from the introduction that Schur's lemma guarantees that $A$ acts on every finite-dimensional simple $R$-module through a character $a:A\to k$. This implies that the set of $k$-points of $i\Laf$ doesn't depend on the choice of $A$. 
	
	In this section, we will prove that without assumptions on $k$, $i\Laf$ as a scheme does not depend on the choice of an $A$. This will follow from a relative version of Schur's lemma. 
	
	\begin{proposition}\label{relschur}
		Let $R$ be a possibly non-commutative ring containing commutative rings $A\subset A'$ in its center. Assume that $R$ is finite as an $A$-module. Let $B$ be an $A$-algebra and $M$ a finite locally free $A$-module equipped with a structure of an $\left(R\otimes_A B\right)$-module such that over every geometric point $b\in \Spec(B)$ over $a\in \Spec(A)$, $M_b$ is a simple $R_a$-module. Then the ring homomorphism $A'\to \en_B(M)$ factors through $B$. 
	\end{proposition}
	
	\begin{proof}
		The homomorphism $R\to \en_B(M)$ is surjective as it is surjective fiberwise over $\Spec(B)$ by the Jacobson density theorem \cite[Chapter XVII, Theorem 3.2]{Lang2002}. It follows that the image of the central subalgebra $A'$ is contained in $B$.  
	\end{proof}

	\begin{proposition}
		Let $R$ be a possibly non-commutative ring containing commutative rings $A\subset A'$ in its center. The natural morphism $\iLaf_{R/A'} \to \iLaf_{R/A}$ is an isomorphism.
	\end{proposition}
	
	\begin{proof}
		It is enough to prove that for any $A$-algebra $B$, every morphism $\Spec(B)\to \iLaf_{R/A}$ can be canonically lifted to a morphism $\Spec(B)\to \iLaf_{R/A'}$. By Proposition \ref{nhilb} the $B$-points of $i\Laf_{R/A}$ are described by finite locally free $A$-modules $M$ satisfying the hypotheses of Proposition \ref{relschur}. By the conclusion, we get that $M$ is also a point of $i\Laf_{R/A'}$.
	\end{proof}
	
	\subsection{Lafforgue's assertion}
	
	By Theorem \ref{main} applied to the case of the Hecke algebra of a compact open subgroup $K\leq G$ as in \S 2.2, we have Lafforgue's assertion \cite[pp. 36]{Lafforgue2016}. The purpose is to consider the Lafforgue variety of the whole reductive group $G$, by taking the inverse limit over all compact open subgroups. As this approach runs into some technical difficulties, we instead apply it to Hecke algebras of Bernstein components to get the following corollary of Theorem \ref{main}.
	
	\begin{corollary}
		For any reductive group $G$ over a non-archimedean local field $F$, we can construct the Lafforgue variety for all smooth irreducible representations of $G$, by taking the union of $\Laf_{H_{\fs}/Z_{\fs}}$ over all Bernstein components $\fs\in \fB(G).$
	\end{corollary}
	
	\begin{remark}
		This approach also has the advantage of allowing us to use the explicit presentations given by Theorem \ref{solleveld} due to Solleveld, which is convenient for applying the methods of the next sections.
	\end{remark}
	
	\section{The Jacobson stratification}
	
	Let $R$ be a possibly non-commutative $k$-algebra such that there is a finitely generated $k$-subalgebra $A$ of the center with $R$ being finite and locally free as an $A$-module. 
	
	Then, by Theorem \ref{main}, the projection $\Laf_{R/A}\rightarrow \Spec(A)$ is finite which implies we can stratify $\Spec(A)$ according to the cardinality of the fiber of $p$. In the case where $\cha(k)=0$, $A$ is regular, and $R$ is a locally free $A$-module, we concretely describe the above stratification using the rank of the Jacobson ideal. 
	

	\subsection{Cohen-Macaulay property}
	
	For a (twisted) affine Hecke algebra, we can choose an appropriate regular subalgebra to apply the Jacobson stratification by using its explicit presentation. In this subsection, we recall that the existence of such a regular subalgebra is guaranteed by a condition that appears more often in the literature. The condition requires that $R$ is a Cohen-Macaulay module (see \cite[\href{https://stacks.math.columbia.edu/tag/0AAH}{Definition 0AAH}]{stacks-project}) over its center $Z$ which is also required to be a Cohen-Macaulay algebra, ie. a Cohen-Macaulay module over itself.

	\begin{proposition}\label{cm}
		Let $R$ be a possibly non-commutative $k$-algebra and $Z(R)$ its center. Assume $Z(R)$ is finitely generated and $R$ is a finite $Z(R)$-module. The following properties are equivalent
		\begin{enumerate}
			\item $R$ is a finite Cohen-Macaulay module over its center $Z(R)$, which is a Cohen-Macaulay $k$-algebra.
			\item $R$ is a finite free $A$-module for a finitely generated regular central subalgebra $A\subseteq Z(R)$.
		\end{enumerate}	
		If $Z(R)$ is regular, we can take $A=Z(R)$.	
	\end{proposition}
	
	\begin{proof}
		This is essentially the same as \cite[Criterion 2.5]{Bernstein1997}.
	\end{proof}
	
	\begin{example}
		The Hecke algebra $\cH_{\fs}(G)$ of a Bernstein component $\fs \in \fB(G)$ satisfies condition (1) \cite[Proposition 3.1]{Bernstein2018}.
	\end{example}
	
	\subsection{Jacobson stratification}
	
	Let $A$ be a commutative ring contained in the center of a possibly non-commutative ring $R$. From now on, we will assume that $R$ is a finite locally free $A$-module. We will also assume that $A$ contains a field $k$ of characteristic zero. 
	
	For every point $a:A\to k(a)$ of $\Spec(A)$, $k(a)$ being a field, the fibre $R_a=R\otimes_A k(a)$ is a finite-dimensional $k(a)$-algebra. The Jacobson radical $J_a=\rad(R_a)$, defined as the intersection of all maximal left ideals of $R_a$, is a 2-sided ideal which can be characterized in multiple ways, namely, it is the intersection of the annihilators of simple left $R_a$-modules, or the maximal left (or right) nilpotent ideals, see \cite[4.2,4.12]{Lam1991}. The quotient $R_a/J_a$ is a semi-simple $k$-algebra which, by the Artin-Wedderburn theorem, is isomorphic to a product of matrix algebras $R_a/J_a=\prod_{i=1}^r M_{n_i}(D_i)$ where $M_{n_i}(D_i)$ is a matrix algebra over a skew field $D_i$ containing $k(a)$ in its center. 
	
	\begin{proposition}
		The function $r_{\mathrm{Jac}}:\Spec(A) \rightarrow \mathbb{N}$ given by $a\mapsto \dim_{k(a)} J_a$ is upper semi-continuous. 
	\end{proposition}
	
	The assertion will follow from yet another interpretation of the Jacobson radical as the kernel of a trace form. We recall that $R$ is a finite locally free as an $A$-module, and for every element $r\in R$, the $A$-linear operator on $R$ given by $x\mapsto rx$ has a well defined trace $\tr_{R/A}(r)$. It follows that we have a symmetric $A$-bilinear form on $R$ given by $Tr_{R/A}(x,y)=\tr_{R/A}(xy)$, or equivalently a $A$-linear map $Tr_{R/A}:R\to R^\vee$. The construction of the trace form and the bilinear form $Tr_{R/A}$ commute in the obvious way with base change and for every geometric point $a:A\to k(a)$, we have a trace form $\tr_a:R_a\to k(a)$ and a symmetric bilinear form $Tr_{R/A,a}$ on $R_a$, or equivalently a linear form $Tr_{R/A,a}:R_a\to R_a^\vee$. 
	
	\begin{proposition}\label{jac}
		For every point $a:A\to k(a)$ of $\Spec(A)$, the Jacobson radical $J_a$ is the kernel of the trace form $Tr_{R/A,a}:R_a\to R_a^\vee$.  
	\end{proposition}
	
	\begin{proof}
		Since $J_a$ is a nilpotent ideal, for every $x\in J_a$ and $y\in R_a$, we have $\tr_{a}(xy)=0$. It follows that $J_a$ is contained in the kernel of $Tr_{R/A,a}$. Moreover, the Artin-Wedderburn theorem implies that $Tr_{R/A,a}$ induces a non-degenerate bilinear form on $R_a/J_a$ and therefore $J_a$ is exactly equal to the kernel of $Tr_{R/A,a}$.
	\end{proof}
	
	We will now construct the stratification of $\Spec(A)$ by the rank of the Jacobson ideal using the concept of determinantal ideals. Assume that $R$ is a locally free $A$-module of rank $n$. Locally for the Zariski topology we may assume that $R$ is a free $A$-module of rank $n$, and the trace form $Tr:R\to R^\vee$ is given by a $n\times n$-matrix. For every positive integer $i$, we define $I_i$ to be the ideal of $A$ such that locally for the Zariski topology, $I$ is generated by the minors to the order $n-i+1$ of the local matrix of $Tr$. We know a chain of inclusions of ideals $0=I_0 \subset I_1 \subset \cdots $ which induces a chain of inclusion of closed subsets $\bar X_0 \supset \bar X_1 \supset \cdots $ where $\bar X_i =\Spec(A/I_i)$. Over $X_i$ the complement of $\bar X_{i+1}$ in $\bar X_i$, the rank of the Jacobson radical is constant of value $i$. 
	
	In fact, over $X_i$ the trace form $Tr_{X_i}: R\otimes_A \cO_{X_i} \to R^\vee \otimes_A \cO_{X_i}$ has kernel a locally free $\cO_{X_i}$-module $J_i$ of rank $i$, and image a locally free $\cO_{X_i}$-module $\bar R_{X_i}$ of rank $n-i$. The trace form $Tr_{X_i}$ induces a non-degenerate symmetric bilinear form $\bar Tr_{X_i}$ on $\bar R_{X_i}$. In particular, for every point $a:A\to k(a)$ of $\Spec(A)$ belonging to the stratum $X_i$, $\bar R_{X_i}\otimes_{\cO_{X_i}} k(a)$ is a semi-simple algebra over $k(a)$. Let $\bar a:A\to \overline{k(a)}$ a geometric point over $a$. Then $\bar R_{X_i} \otimes_{\cO_{X_i}} \overline{k(a)}$ is isomorphic to a product of matrix algebras
	$$ \bar R_{X_i}\otimes_{\cO_{X_i}} \overline{k(a)}= \prod_{i=1}^r M_{n_i}(\overline{k(a)})
	$$
	where $\underline n(a)=(n_1,\ldots,n_r)$ is unordered sequence of positive integers depending only on $a$.

	\begin{proposition}\label{strat}
		The function $a\mapsto \underline n(a)$ is locally constant on $X_i$. In particular, $a\mapsto r(a)$ is a locally constant function on $X_i$.
	\end{proposition} 
	
	\begin{proof}
		Since $\bar R_{X_i}$ is a locally free $\cO_{X_i}$-module equipped with a structure of an associative algebra which is fiberwise semisimple over $X_i$, its invertible elements define a group scheme $\cG_{\bar R_{X_i}}$ over $X_i$ by Proposition \ref{groupsch}. The geometric fiber over a geometric point $\bar a$ is isomorphic to $\GL_{n_1}\times \cdots \times \GL_{n_r}$. Thus $\cG_{\bar R_{X_i}}$ is a smooth reductive group scheme and the function  $a\mapsto \underline n(a)$ is locally constant by \cite[Expos\'e XIX, Corollaire 2.6]{Artin1963}.
	\end{proof}

	\section{Generalized discriminants and irreducibility of induced representations}
	
	Let $R$ be a possibly non-commutative $k$-algebra that is a free finite module over a finitely generated subalgebra $A$ of its center $Z(R)$. We assume $k$ to be an algebraically closed field of characteristic $0$.
	
	In this section, we focus on the open dense stratum $X_0\subseteq \Spec A$. If the Jacobson radical of $R$ is trivial, $X_0$ is the semisimplicity locus of $R$. We characterize $X_0$ as the complement of the zero set of a generalized discriminant ideal.  
	
	For the Hecke algebra $\cH_{\fs}$ of a Bernstein component $\fs\in \fB(G)$. Then, the projection $\Laf_{\cH_{\fs}/Z_{\fs}}\rightarrow \Spec(Z_{\fs})$ sends a smooth irreducible representation $\rho$ to its cuspidal support $\cs(\rho):=(M,\sigma)$, defined by $\rho\hookrightarrow i_M^G(\sigma)$. For a generic $(M,\sigma)$, the induced representation $i_M^G(\sigma)$ is irreducible. As a main application of the results in this section, we prove Theorem \ref{ind} providing a computational criterion for the irreducibility of $i_M^G(\sigma)$ outside a singular locus.
	
	We provide computational tools for the discriminant in cases where there is an explicit presentation as in the case of (twisted) affine Hecke algebras. In particular, we compute the discriminant for the Iwahori-Hecke algebra of a split reductive $p$-adic group, first for the case of an adjoint group where the center is already regular, and then for the general case, where we need to choose a regular subalgebra. 
	
	\subsection{Definition and properties}
	
	If $f: R\rightarrow R'$ is a map of free $A$-modules of rank $n$, the $n$-th exterior power $\nex f : \nex R \rightarrow \nex R'$ is a map of free $A$-modules of rank $1$ \cite[Chapter 7, Theorem 8.1]{Bourbaki1989}. In the case $R'=R$, by means of the canonical isomorphism $\en_A(A)\cong A$ given by the inverse homomorphisms $a\rightarrow r_a(x)=ax$ and $r\rightarrow r(1)$, we can canonically associate to $f: R\rightarrow R$ its determinant $\det(f)$.
	
	From now on we assume $R$ is also an $A$-algebra.
	
	\begin{definition}\label{norm}
		The norm function $N_{R/A}: \en_A(R)\rightarrow A$ is the map sending an endomorphism $f\in \en_A(R)$ to $N_{R/A}(f):=\det(f).$ If $r\in R$, by associating to $r$ the endomorphism $h_r\in \ed_A(R)$ given by $h_r(x)=rx$, we also define $N_{R/A}(r):=N_{R/A}(h_r).$
	\end{definition}
	
	We recall some elementary properties of the norm.
	
	\begin{lemma}\label{normprop}
		For the norm function $N_{R/A}$ we have
		\begin{itemize}
			\item $N_{R/A}(fg)=N_{R/A}(f)N_{R/A}(g),$
			\item If $a\in A$, $N_{R/A}(a)=a^n$ where $n$ is the rank of $R$ over $A$.
		\end{itemize}
	\end{lemma}
	
	\begin{proof}
		The first part follows from multiplicativity of the determinant. For the second, $a$ can be identified with a scalar matrix.
	\end{proof}
	
	The next Lemma is less trivial than it may appear, for a proof, see \cite[Appendix B, Lemma 4]{Cassels1986}.	
	
	\begin{lemma}
		If $B$ is a commutative $A$-algebra that is free as an $A$-module and $C$ is a $B$-algebra such that $C$ is locally free over $B$, we have
		$$N_{C/A}=N_{B/A}\circ N_{C/B}.$$
		In particular, if $n$ is the rank of $C$ over $B$, $$N_{C/A}(b)=\left( N_{B/A}(b) \right)^n$$
	\end{lemma}
	
	When $R'\neq R$, we can only identify $\nex f$ with an element of $A$ after choosing bases. Nonetheless, if $R'=R^{\vee}$, a basis $b=\{r_1,\ldots, r_n\}$ of $R$ as an $A$-module uniquely determines a dual basis $b^{\vee}$ of $R^{\vee}:=\Hom_A(R,A)$. By the universal property of free modules, we now have canonical identifications 
	
	\begin{eqnarray*}
		\nex R \cong A\\
		\nex R^{\vee} \cong A\\
		\left(\nex f\right)_{b}\in \en_A(A)\cong A
	\end{eqnarray*}
	
	If $\bar{b}=Mb$ is another basis for $R$, then $\det(M)$ is an invertible element of $A$, and the dual basis is given by $\bar{b}^{\vee}:={}^tMb^{\vee}.$ Thus,
	$$\left(\nex f\right)_{\bar{b}}=\det(M)\cdot \left(\nex f\right)_{b} \cdot \det({}^tM)=\det(M)^2\cdot \left(\nex f\right)_{b}$$
	
	\begin{definition}\label{disc}
		Let $Tr_{R/A}: R\otimes_A R\rightarrow A$ be the trace form defined by $Tr_{R/A}(x,y)=tr_A(xy)$. We consider it as a function $Tr_{R/A}: R\rightarrow R^{\vee}$. By the preceding paragraph, all possible choice of a basis $b$ for $R$ define the same element $\left(\nex Tr_{R/A}\right)_b$ up to multiplication by $A^{\times}$, and therefore generate the same locally principal ideal $d_{R/A}.$
		
		We call $d_{R/A}$ the discriminant of $R$ over $A$.
	\end{definition} 
	
	\begin{remark}
		In the case where $R$ is free over $A$, any choice of a generator for $d_{R/A}$ provides us with the same regular function on $\Spec(A)$, so we often treat $d_{R/A}$ as a function.
	\end{remark}
	
	\begin{remark}
		In the case of number rings, Definition \ref{disc} agrees with the classical discriminant of algebraic number theory \cite[\S III.6]{Serre1979}. 
	\end{remark}
	
	As we mainly use the Jacobson stratification over the open dense stratum $X_0$, our definition is motivated by the following lemma.
	
	\begin{lemma}\label{vanish}
		The open stratum $X_0$ in the Jacobson stratification of $\Spec A$ is the complement of the zero set $V(d_{R/A})$
	\end{lemma}
	
	\begin{proof}
		By Definition \ref{disc}, the zero set is exactly the locus where the trace form is an isomorphism, thus the locus where the Jacobson radical is trivial by Proposition \ref{jac}.
	\end{proof}
	
	In the case of number fields, we have an elementary formula allowing us to compute the discriminant of a tower of extensions in terms of the discriminants of the intermediate steps. As stated in the introduction, it turns out it can be generalized to our case. We recall \Cref{transd}.
	
	\transd*
	
	\begin{proof}
		Let $X=\Spec A, Y=\Spec B, Z=\Spec C$ and $g:Z \rightarrow Y$ and $f:Y\rightarrow X$ the maps corresponding to inclusion.
		
		Let $R_{Y/X}$ be the ramification divisor for $f$. By reducing to the local case, $f_*R_{Y/X}=div(d_{B/A})$. We use the relative short exact sequence of K{\"a}hler differentials, where injectivity follows by the fact that all maps are smooth
		$$0\rightarrow \Omega_{B/A}\otimes_B C\rightarrow \Omega_{C/A}\rightarrow \Omega_{C/B}\rightarrow 0$$		
		We take determinants in the sense of \cite[Exercise II.6.11]{Hartshorne1977}, to get
		$$\det(\Omega_{C/A})\cong \det(\om_{B/A}\otimes_B C)\otimes \det(\om_{C/B}).$$		
		Now by the smoothness of $f,g$ we have  $\det(\om_{C/A})=\omega_{C/A}$, $\det(\om_{B/A})=\omega_{C/B}$ and thus $\det(\om_{B/A}\otimes_B C)=g^*\omega_{B/A}.$
		Therefore,
		$$\omega_{C/A}\cong \omega_{C/B}\otimes g^*\omega_{B/A}.$$
		We know that $\omega_{C/A}=\mathcal{L}(R_{Z/X})$ is the invertible sheaf corresponding to the ramification divisor $R_{Z/X}$. Thus, taking associated divisors,		
		$$ R_{Z/X} \cong R_{Z/Y}+g^*R_{Y/X}$$		
		We consider the pushforward by $f\circ g$ to get the divisor corresponding to the discriminant. Since $g_*g^*$ for divisors is multiplication by the degree, and $f_*div(z)=div(N(z))$, we get 
		\begin{eqnarray*}
			div(d_{C/A}) & \cong & f_*div(d_{C/B})+f_*(nR_{Y/X})\\
			& \cong &  div(N_{B/A}(d_{C/B}))+div((d_{B/A})^n)\\
			& \cong &  div((d_{B/A})^n N_{B/A}(d_{C/B}))\\
		\end{eqnarray*}
	\end{proof}
	
	\begin{remark}
		We can also deduce Lemma \ref{transd} by repeated application of the generalized Riemann-Hurwitz formula.
		
		Indeed, we have
		\begin{eqnarray*}
			R_{Z/X} & \cong & K_Z-(f\circ g)^*K_X \\
			& \cong & K_Z-g^* f^*K_X \\
			& \cong & K_Z-g^*(K_Y-R_{Y/X})\\
			& \cong & R_{Z/Y}+g^*R_{Y/X}\\
		\end{eqnarray*}
		and we conclude as before.
	\end{remark}

	\subsection{Irreducibility of induced representations}
	
	In the case of the Hecke algebra $\cH_{\fs}$ of a Bernstein component $\fs \in \fB(G)$, we know after \cite[Proposition 3.1]{Bernstein2018} that $\cH_{\fs}$ is a finite Cohen-Macaulay module over its center $Z_{\fs}$ which is itself a Cohen-Macaulay algebra. We can apply the Lafforgue variety construction for $R=\cH_{\fs}$ and $A=Z_{\fs}$, but using Proposition \ref{cm} we can also apply it for $A$ being a regular algebra contained in $Z_{\fs}$ such that $Z_{\fs}$ is a finite $A$-module. If $A$ is regular, then both $\cH_{\fs}$ and $Z_{\fs}$ are finite locally free $A$-modules. In this case, the group scheme $\cG_{R/A}$ is smooth acting on the Hilbert scheme $\Hilb_{R/A}$.
	
	If $f:\Spec(Z_{\fs})\rightarrow \Spec(A)$ is the projection corresponding to the inclusion $A\subseteq Z_{\fs}$, we define $Z(f)$ to be the closed subset of $\Spec(Z_{\fs})$ where $f$ is not smooth. Let $X_0$ be the open dense stratum of $\Spec(Z_{\fs})$ given by the cardinality of the fiber of the projection from the Lafforgue variety. We identify a cuspidal datum $(M,\sigma)$ with the corresponding point in $\Spec(Z_{\fs})$.
	
	We recall and prove \Cref{ind}.
	\ind*
	
	\begin{proof}
		By Theorem \ref{main}, we get finite projections
		
		$$\Laf_{\cH_{\fs}/Z_{\fs}}\cong \Laf_{\cH_{\fs}/A}\xrightarrow{p} \Spec(Z_{\fs})\xrightarrow{f} \Spec(A)$$ 
		
		By the definition of $p$, the Jordan-H\"older constituents of the induction correspond to the points in the fiber, and thus  $|JH(i_M^G(\sigma))|=|p^{-1}(M,\sigma)\cap \iLaf_{\cH_{\fs}/A}|$. Since generically an induced representation is irreducible, over $X_0$ the cardinality of the fiber is $1$ which proves the first assertion.
		
		Let $Y_0$ be the open dense stratum of the Jacobson stratification for $\Spec(A)$, and $n=deg(f)$. Then, for a generic point $a\in \Spec(A),$ the cardinality of $f\circ p$ is $n$, and thus the cardinality of the fiber over the point $a\in \Spec(A)$ is $\geq n$ with equality if and only if $a\in Y_0$. 
		
		Since the fibers of $f$ outside the singular locus have cardinality $n$, Lemma \ref{vanish} implies the second assertion.
	\end{proof}
	
	\begin{example}
		Let $H$ be the Iwahori-Hecke algebra of $G=\GL_2(F)$. We consider the split torus $T$ consisting of diagonal matrices and the Weyl group $W\cong S_2$. By Proposition \ref{beriha}, $H$ is generated over the group algebra of the cocharacter lattice $R\cong \mathbb{C}[x_1^{\pm},x_2^{\pm}]$ by $T_e=1$ and an element $T_s$ satisfying 
		\begin{eqnarray}
			(T_s+1)(T_s-q)=0 \label{quad}\\
			T_sx_1=x_2T_s+(q-1)x_1 \label{in}
		\end{eqnarray}
		The center is $R^W=\mathbb{C}[x_1^{\pm},x_2^{\pm}]^{S_2}$. A cuspidal datum in this case corresponds to a choice of an unordered pair of complex numbers defining an unramified character of $T$.
		
		Let $V$ be a simple $H$-module. The subalgebra $R$ is abelian therefore we can choose a common eigenvector $v\in V$. By Equation \ref{in}, every element $h\in H$ can be written $h=T_er_1+T_sr_2$ for $r_1,r_2\in R$. Since $v$ is cyclic by simplicity of $V$, $\dim(V)\leq 2$. By Equation \ref{quad}, if $\dim(V)=2$, then $tr_{T_s}(V)=q-1$, and if $\dim(V)=1$ then $T_s$ acts as either $-1$ or $q$.
		
		Therefore, the trace ring is $T_H=\mathbb{C}[x_1^{\pm},x_2^{\pm}]^{S_2}\oplus \mathbb{C}[x_1^{\pm}]\oplus \mathbb{C}[x_1^{\pm}]$. The Lafforgue variety and the projection are therefore roughly given by Figure \ref{fig:key}.
		
		\begin{figure}[h]
			
			\centering
			
			\begin{tikzpicture}
				\draw (0,3.9)--(3,3.9)--(4,5)--(1,5)--(0,3.9);
				\draw (0,5.4)--(4,6.5);
				\draw (0,2.7)--(4,3.8);
				\draw[dashed] (0,3.9)--(4,5);
				\draw[thick,->] (2,2.8)--(2,1.8);
				\draw (0,0.4)--(3,0.4)--(4,1.5)--(1,1.5)--(0,0.4);
			\end{tikzpicture}
			\caption{Projection from the Lafforgue variety to the Bernstein variety for Iwahori representations of $GL_2(F)$}
			\label{fig:key}
		\end{figure}
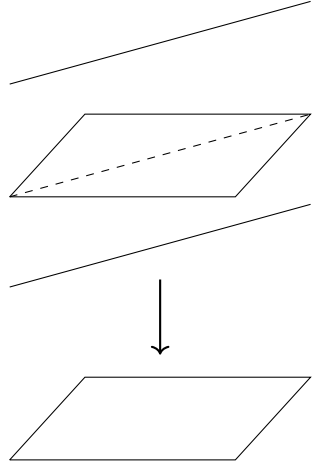
		In this case, the center $R^W$ is already regular, and the discriminant is
		$$d_{H/R^W}=(x_2-qx_1)^2(x_1-qx_2)^2,$$
		which retrieves that the induction $i_{M}^G(\chi_1,\chi_2)$ is irreducible if and only if $\chi_1\chi_2^{-1}\neq q^{\pm}$. When this is not the case, the Jordan-Holder constituents of the induction are an irreducible character and a Steinberg representation corresponding to the two other connected components shown in Figure \ref{fig:key}.
	\end{example}
	
	\subsection{Discriminant of adjoint reductive groups}
	
	The center $R^W$ of the Iwahori-Hecke algebra is also the coordinate ring of $\hat{T}\git W$ where $\hat{T}$ is the dual torus. In this subsection we assume $G$ is an adjoint group, thus $\hat{T}$ is the torus of a simply-connected group, and in this case $\hat{T}\git W\cong \bA^r$ where $r$ is the rank of $G$. Thus, for an adjoint group $R^W$ is regular. In this case, we can retrieve Kato's result by computing $d_{H/R^W}$.	
	
	The computation essentially will be performed in two steps, from $H$ to $R$ and from $R$ to $R^W$, in a similar fashion to Lemma \ref{transd}, which cannot be used directly since $H$ is non-commutative. It turns out that the discriminant behaves in a similar way nonetheless.
	
	Let $W=\{w_1,\ldots, w_n\}$ and $I_{w_i}, K_{w_i}$ be the intertwiners/ normalized intertwiners as we defined them in Section 2. Then, $d_{H/R^W}$ is the discriminant of the lattice $\{I_{w_i}\pi^{\mu_j}\}_{i,j}$ for proper $\mu_j$.
	
	Thus, we need to compute the determinant $$\det(\{tr(I_{w_i}\pi^{\mu_j} I_{w_k}\pi^{\mu_l})\}_{i,j,k,l\in [n]})=\det(\{tr(I_{w_i} I_{w_k}\pi^{w_k(\mu_j)}\pi^{\mu_l})\}_{i,j,k,l\in [n]}).$$
	
	We notice that for $w_i\neq w_k^{-1}$ the trace is zero, because elements $I_w$ for $w\neq e$ permute the generelized eigenvectors, so the matrix has $n$ blocks of size $n\times n$. Also, we recall that setting $e_a=1-q^{-1}\pi^{a^{\vee}}, d_a=1-\pi^{a^{\vee}}$, gives
	$$I_wI_{w^{-1}}=\prod_{a\in R_w} \frac{e_ae_{-a}}{d_ad_{-a}}$$
	by Lemma \ref{int}.
	
	Thus, we can simplify the calculation using the following.
	
	\begin{lemma}\label{nctrans}
		Let $R$ be a commutative algebra over the commutative algebra $A$. Let $p, r_1,\ldots, r_n\in R.$ Then we have that
		$$\det(\{tr(pr_ir_j)\})=N_{R/A}(p)\cdot \det(\{tr(r_ir_j)\}).$$
	\end{lemma}
	
	\begin{proof}
		Consider a basis of generalized eigenvectors $v_i$ and let $\kappa_i$ be the eigenvalues of $p$ and $\lambda_i^j$ be the eigenvalues of $r_j.$ Then $tr(pr_ir_j)=\sum \kappa_k\lambda_k^i\lambda_k^j.$
		We therefore have
		\[
		\begin{pmatrix}
			tr(pr_1^2) & \ldots & tr(pr_1r_n) \\
			\vdots & \ddots & \vdots \\
			tr(pr_nr_1) & \ldots & tr(pr_n^2)
		\end{pmatrix}=
		\begin{pmatrix}
			\kappa_1\lambda_1^1 & \ldots & \kappa_1\lambda_1^n \\
			\vdots & \ddots & \vdots \\
			\kappa_n\lambda_n^1 & \ldots & \kappa_n\lambda_n^n
		\end{pmatrix}\cdot
		\begin{pmatrix}
			\lambda_1^1 & \ldots & \lambda_n^1 \\
			\vdots & \ddots & \vdots \\
			\lambda_1^n & \ldots & \lambda_n^n
		\end{pmatrix}
		\]		
		The above product is equal to
		\[
		\kappa_1\cdots \kappa_n\cdot \begin{pmatrix}
			\lambda_1^1 & \ldots & \lambda_1^n \\
			\vdots & \ddots & \vdots \\
			\lambda_n^1 & \ldots & \lambda_n^n
		\end{pmatrix}\cdot
		\begin{pmatrix}
			\lambda_1^1 & \ldots & \lambda_n^1 \\
			\vdots & \ddots & \vdots \\
			\lambda_1^n & \ldots & \lambda_n^n
		\end{pmatrix}=
		\det(p)\cdot 
		\begin{pmatrix}
			tr(r_1^2) & \ldots & tr(r_1r_n) \\
			\vdots & \ddots & \vdots \\
			tr(r_nr_1) & \ldots & tr(r_n^2)
		\end{pmatrix}
		\]
	\end{proof}	
	By Lemma \ref{nctrans}, we get 
	$$d_{H/R^W}=\left(\prod_{i=1}^n\prod_{a\in R_{w_i}} N_{R/R^W}\left(\frac{e_ae_{-a}}{d_ad_{-a}}\right)\right)\cdot d_{R/R^W}^n$$
	Notice that we also have $$\prod_{i=1}^n\prod_{a\in R_{w_i}} \frac{e_ae_{-a}}{d_ad_{-a}}=\prod_{a\in \Phi} \prod_{w\in W, a\in R_w} \frac{e_ae_{-a}}{d_ad_{-a}}=\left(\prod_{a\in \Phi}\frac{e_ae_{-a}}{d_ad_{-a}}\right)^{n/2},$$
	so, since this element is $W$-invariant and thus by Lemma \ref{normprop} its norm is itself to the $n$-th power, we get the general formula	
	\begin{equation}\label{eqn}
		d_{H/R^W}=\left(\prod_{a\in \Phi}\frac{e_ae_{-a}}{d_ad_{-a}}\right)^{n^2/2}\cdot d_{R/R^W}^n.
	\end{equation}	
	We can compute the discriminant for $R/R^W$. Indeed, it is enough to calculate the ramification divisor of the map $\Spec R\rightarrow \Spec R^W$. Ramification happens when $d_a=0$ for some $a$. Indeed, in that case, the corresponding homomorphism $R\rightarrow k$ was $s_a$-invariant, and in that case two sheets degenerate in one in every ramified point. Thus, $d_a$ appears with an exponent of $1$ in the ramification divisor. Pushed forward, we have that the discriminant of the extension $R/R^W$ is $\left(\prod_{a\in \Phi} d_a\right)^n.$ This computation is also carried out algebraically by Steinberg \cite[pp. 125-127]{Steinberg1974}.
	
	Combined with the fact that $d_a=d_{-a}$ up to an invertible element, (\ref{eqn}) shows that for an adjoint group $G$ we have
	
	\begin{proposition}
		If $G$ is adjoint, we have 
		$$d_{H/R^W}=\left(\prod_{a\in \Phi} e_ae_{-a}\right)^{n^2/2}.$$
	\end{proposition}
	
	\begin{remark}
		Since the zero locus of $d_{H/R^W}$ is exactly the locus where the induced representation is reducible by the considerations in Section 3, for the case of an adjoint group we retrieve Kato's result  \cite[Theorem 2.2]{Kato1981}. Notice that for an adjoint group the second condition of Kato's theorem is always true.
	\end{remark}

	\subsection{Discriminant in the non-adjoint case}
	
	If $G$ is not adjoint, $R^W$ is not regular anymore, so we need to restrict to some subalgebra $A$ that is regular. We make a canonical choice.
	
	\begin{definition}
		We identify the fundamental weights $\omega_1,\ldots,\omega_n$ with the trace function of the corresponding fundamental representations. Then, we consider the smallest integers $d_1,\ldots, d_n$ such that $\omega_i^{d_i}\in R^W$. We call $A=\mathbb{C}[\omega_1^{d_1},\ldots, \omega_n^{d_n}]\subseteq R^W$ the algebra of fundamental weights.
	\end{definition}
	
	$A$ is obviously regular as it is a polynomial algebra. 
	
	By the same procedure as in the previous subsection, equation \ref{eqn} is still true upon replacing $d_{R/R^W}$ by $d_{R/A}$, so we want to compute $d_{R/A}$. Recall that $R$ is the group algebra of the cocharacter lattice, so alternatively, it is the function ring of the dual torus $\hat{T}$. The fact that $R^W$ is regular when $G$ is adjoint comes from the fact in the adjoint case it is known that $\hat{T}\git W$ is an affine space of dimension $\mathrm{rank}(G)$, and it is given as polynomials over the trace functions corresponding to the fundamental weights.
	
	For the general case, we consider a simply connected cover $\tilde{T}$ of $\hat{T}$ such that $\hat{T}=\tilde{T}\git Z$. We define $R^+=k[\tilde{T}]$. Then $R^+$ is regular and $(R^+)^W$ is also regular.
	
	Consider the following diagram
	\[
	\begin{tikzcd}[every arrow/.append style={dash}]
		& R\arrow[r]\arrow[d] & R^+\arrow[d]\\
		& R^W\arrow[r]\arrow[ld] & (R^+)^W\\
		A & & 
	\end{tikzcd}
	\]
	We want to compute $d_{R/A}$. By Lemma \ref{transd}, it is enough to compute the discriminants
	$$d_{R^+/R}, d_{R^+/(R^+)^W}, d_{(R^+)^W/A}.$$
	We already know $d_{R^+/(R^+)^W}$. Since $\hat{T}\git W\cong \C[\omega_1,\ldots, \omega_n]$, we have 
	$$d_{(R^+)^W/A}\cong \omega_1^{d_1(d_1-1)}\cdots \omega_n^{d_n(d_n-1)}.$$
	We also have that $N_{R/A}(d_{R^+/R})$ is an invertible element. Let $n=|W|$ as before, and $[R^+:R]=r$. Then if we set $[R:A]=nd$ we have $[(R^+):A]=rd$. By using Lemma \ref{transd} we get the following theorem.
	
	\begin{proposition}\label{gendisc}
		In the general case, and for $A$ being the algebra of fundamental weights, we have
		$$d_{H/A}=\left(\prod_{a\in \Phi} e_ae_{-a}\right)^{dn^2/2}\cdot \left(\omega_1^{d_1(d_1-1)}\cdots \omega_n^{d_n(d_n-1)}\right)^{n/r}.$$
	\end{proposition}
	
	\begin{proof}
		By the same method as in the adjoint case,
		$$d_{H/A}=\left(\prod_{a\in \Phi}\frac{e_ae_{-a}}{d_ad_{-a}}\right)^{dn^2/2}\cdot d_{R/A}^n.$$
		By Lemma \ref{transd}, writing the discriminant $d_{R^+/A}$ in two different ways and ignoring the invertible factor $N_{R/A}(d_{R^+/R})$ we have
		$$d_{R/A}^r=(d_{R^+/(R^+)^W})^{rd}\cdot  d_{(R^+)^W/A}^n.$$
		Combining the two equations with $d_{R^+/(R^+)^W}=\left(\prod_{a\in \Phi} d_a\right)^n$ gives the result.
	\end{proof}
	
	\begin{example}[$SL_2$ case]
		For $SL_2$ the dual group is $PGL_2$, and the simply connected cover would be again $SL_2$. This gives $R^+=\mathbb{C}[x^{\pm}]$, while $ R^+\git \mathbb{Z}_2=\mathbb{C}[x^{\pm 2}].$ Then $(R^+)^W\cong  \mathbb{C}[x+x^{-1}]$ and $R^W=\mathbb{C}[x^2+x^{-2}]=A$ - this is the only simply connected group for which that is correct, since $\omega_1=x+x^{-1}$ so $\omega_1^2$ generates $R^W$.
		
		It is easy now to compute directly $d_{R^+/R}=x^2, d_{R^+/(R^+)^W}=(1-x^{-2})^2,  d_{(R^+)^W/A}=(1+x^{-2})^2$, which gives (one can also do this directly to get the same result) $d_{R/A}=(1-x^{-4})^2$ since $x^2$ is invertible. 
		
		Therefore, either by Proposition \ref{gendisc} or by direct computation, 
		
		$$d_{H/R^W}=(1-q^{-1}\pi^{a^{\vee}})^2\cdot (1-q^{-1}\pi^{-a^{\vee}})^2\cdot (1+\pi^{a^{\vee}})^4.$$
		
		As a corollary, the induced representation is irreducible if and only if one of the three factors is zero. The same result can be obtained from Kato's theorem or a direct calculation of the conditions \cite[pp. 1020]{Solleveld2021}.
	\end{example} 
	
	\appendix
	\section{The nested Quot functor}
	\subsection{Hilbert and Quot schemes}
	
	Recall that for a scheme $S$ and a flat coherent sheaf $\f$ on $\bP_S^r$ the Hilbert polynomial $h(\f,t)$ of the fiber over a point $s\in S$ is independent of $s$, see \cite[Proposition 4.2.1]{Sernesi2006}. As all functors we will consider will end up being representable, we will abuse notation and denote in the same way the functors with the schemes they represent.
	
	For a projective scheme $X$ over $S$, a coherent sheaf $E$ on $X$ and a polynomial $P(t)$, the Quot functor parametrizes flat families of quotients of $E$ with Hilbert polynomial $P(t)$ over a test scheme $Z\maps S$. More precisely, we have the following.
	
	For a coherent sheaf $E$ on $X$ and a test $S$-scheme $Z$, we denote $X_Z:=X\times_S Z$, $p_X:X_Z\maps X, p_Z:X_Z\maps Z$ the projections, and $E_Z:=p_X^* E.$ Given two quotients $q_i:E_Z\twoheadrightarrow \f_i, i=1,2$ we call them equivalent and write $\f_1\sim \f_2$ if there exists $a:\f_1\xrightarrow{\cong}\f_2$ such that $a\circ q_1=q_2$.
	
	\begin{definition}
		The Quot functor $\Quot_{E,P(t)}^{X/S}:(Sch/S)\rightarrow Sets$ assigns to a test scheme $Z\maps S$ the set of flat coherent quotients $E_Z\twoheadrightarrow\f$ of $E_Z$ and with Hilbert polynomial $P(t)$ on the fibers of $p_Z: X_Z\maps Z$ up to equivalence, ie.
		
		$$\Quot^{X/S}_{E,P(t)}(Z):=\{ E_Z\twoheadrightarrow \f, \f \ \mathrm{is \ flat \ over}\  Z  \ \mathrm{and} \ h(\f,t)=P(t)\}/ \sim,$$
	\end{definition}
	
	A proof of the following theorem by Grothendieck \cite[Théorème 3.1]{Grothendieck1961} can also be found in \cite[Theorem 4.4.1]{Sernesi2006}. It proceeds by embedding the Quot scheme in a suitable Grassmannian.
	
	\begin{theorem}\label{quot}
		$\Quot^{X/S}_{E,P(t)}(Z)$ is representable by a projective scheme over $S$.
	\end{theorem}
	
	Consider now the Hilbert functor parametrizing flat families of closed subvarieties of $X$, ie.
	$$\Hilb^{X}_{P(t)}(Z):=\{Y \subseteq X_Z\}.$$
	We view $Y$ as the flat family $Y_z\subseteq X$ of the fibers of the projection $Y\maps Z$ over points $z\in Z$.
	
	\begin{proposition}
		The functor $\Hilb^{X}_{P(t)}$ is representable by a projective scheme over $S$.
	\end{proposition}
	
	\begin{proof}
		For a closed subscheme $Y\subseteq X_Z$ we denote by $I_Y$ the ideal sheaf.
		
		We have $\Hilb^{X}_{P(t)}\cong\Quot^{X/\Spec k}_{\cO_X,P(t)},$ by the map $Y\subseteq X_Z\maps \cO_X\rightarrow \cO_X/I_Y$, and thus the proposition follows by applying Theorem \ref{quot}.
	\end{proof}
	
	Hilbert schemes, and therefore Quot schemes, can be very exotic as spaces. It is not necessarily reducible or equidimensional, even for fairly simple examples.
	
	\begin{example}
		The Hilbert scheme $\Hilb^{\bP^3}_{2(t+1)}$ is the union of a component of dimension 8 parametrizing pairs of disjoint lines and a component of dimension 11 parametrizing a pair of a conic and a point, see \cite[\S 4.6.3]{Sernesi2006}. These two components intersect at a point $X$ as in \cite[Example 4.2.3(ii)]{Sernesi2006}.
	\end{example}
	
	\begin{definition}\label{quotfull}
		For a projective scheme $X\subseteq \bP^r$, we define the Hilbert scheme of $X$ to be the union over all numerical polynomials $P(t)$ of the individual Hilbert schemes,
		
		$$\Hilb^X:=\bigsqcup_{P(t)} \Hilb^X_{P(t)}.$$
		
		Similarly for the Quot scheme we define
		
		$$\Quot_{E}^{X/S}=\bigsqcup_{P(t)}\Quot_{E,P(t)}^{X/S}$$
	\end{definition}
	
	Following \cite[\S 4.6.4]{Sernesi2006}, for a coherent sheaf $E$ on a scheme $S$, we notice that Theorem \ref{quot} implies that for the constant polynomial $P(t)=n$ we have the following.
	
	\begin{theorem}\label{quotf}
		The functor $\Quot_n(E):=\Quot_{E,n}^{S/S}$ sending a test $S$-scheme $T$ to the flat families of locally free quotients of $E$ over $T$ is representable by a projective $S$-scheme.
		
		In particular, for $E$ also locally free, the \emph{relative Grassmannian} $G_n(E):=\Quot_n(E^{\vee})$ classifying locally free subbundles of rank $n$ is a projective smooth $S$-scheme.
	\end{theorem}
	
	\begin{proof}
		Since the identity morphism is projective, $X=S$ is projective over $S$ and Theorem \ref{quot} applies. The smoothness part of the latter assertion follows from the observations above \cite[Proposition 4.6.1]{Sernesi2006}.
	\end{proof}

	\subsection{The nested Quot functor}
	
	Let $S$ be a scheme, and $X$ a scheme over $S$. As before, for a coherent sheaf $E$ on $X$ and a test $S$-scheme $Z$, we denote $X_Z:=X\times_S Z$, $p_X:X_Z\maps X, p_Z:X_Z\maps Z$ the projections, and $E_Z:=p_X^* E.$
	
	Let $\bar{P}(t)=(P_1(t),\ldots, P_n(t))$ be a sequence of numerical polynomials.
	
	\begin{definition}\label{nquotdef}
		The nested Quot functor $\mathrm{F}_n\Quot_{E,\bar{P}(t)}^{X/S}:(Sch/S)\rightarrow Sets$ assigns to a test scheme $Z\maps S$ the set of flags of flat quotients $\{E_T\twoheadrightarrow \f_1 \twoheadrightarrow \ldots \twoheadrightarrow \f_n\}$ of $E_Z$ and with Hilbert polynomial $P_i(t)$ on the fibers of $p_Z:X_Z\maps Z$ up to compatible isomorphism, ie.
		
		$$\Quot^{X/S}_{E,P(t)}(Z):=\{ E_Z\twoheadrightarrow \f_1 \twoheadrightarrow \ldots \twoheadrightarrow \f_n, \f_i \ \mathrm{is \ flat \ over}\  Z  \ \mathrm{and} \ h(\f_i,t)=P_i(t)\}/ \sim,$$
	\end{definition}
	
	\begin{theorem}\label{nquot}
		The nested Quot functor $\mathrm{F}_n\Quot_{E,\bar{P}(t)}^{X/S}:(Sch/S)\rightarrow Sets$ parametrizing flat families of nested quotients of $E$ of length $n$, is representable by a projective scheme over $S$.
	\end{theorem}
	
	Let $FQ_n:=\mathrm{F}_n\Quot_{E,\bar{P}(t)}^{X/S}.$
	
	The idea of the proof, as in \cite[\S 1.1]{Monavari2022}, is that we have a natural morphism of functors 
	\begin{eqnarray*}
		FQ_n(Z)& \maps & \prod_{i=1}^n \Quot^{X/S}_{E,P_i(t)}(Z)\\
		(E_Z\twoheadrightarrow \f_1 \twoheadrightarrow \ldots \twoheadrightarrow \f_n) & \mapsto & \left(E_Z\twoheadrightarrow \f_1, E_Z\twoheadrightarrow \f_2,\ldots, E_Z\twoheadrightarrow \f_n \right)
	\end{eqnarray*}
	Let $\phi_i: E_Z\twoheadrightarrow \f_i$ be the $i$-th component of the right hand side. Then the image is given by the closed condition $$\ker \phi_1\hookrightarrow \ker \phi_2\hookrightarrow \ldots \hookrightarrow \ker \phi_n.$$
	
	\begin{proof}
		For the case $n=1$, this is Grothendieck's classical Quot scheme, \cite[Théorème 3.1]{Grothendieck1961}, see another proof in \cite[Theorem 4.4.1]{Sernesi2006}. We will show it for the case $n=2$ and the rest follows by induction.
		
		Let $Q_i=\Quot^{X/S}_{E,P_i(t)}$, $i=1,2$, $\un_i$ the universal quotients corresponding to the identity morphism under the identification of $\Hom(Q_i,Q_i)=Q_i(Q_i)$, and $Q=Q_1\times_S Q_2$. Let $p_i$ be the projections.
		
		We consider the injective morphism of functors 
		\begin{eqnarray*}
			FQ_2  \maps  Q 
		\end{eqnarray*}
		given by
		\begin{eqnarray*}
			FQ_2(Z)& \maps & Q(Z)\\
			(E_Z\twoheadrightarrow \f_1 \twoheadrightarrow \f_2) & \mapsto & \left(E_Z\twoheadrightarrow \f_1, E_Z\twoheadrightarrow \f_2\right)
		\end{eqnarray*}
		
		We have short exact sequences	
		\begin{eqnarray*}
			0 \maps N_i\maps E_{Q_i}\maps \un_i \maps 0
		\end{eqnarray*}
		which after pullback and by observing $p_1^*E_{Q_2}=p_2^*E_{Q_1}=E_Q$ provide us sequences
		\begin{eqnarray*}
			0 \maps p_i^*N_i\maps E_{Q}\maps p_i^*\un_i \maps 0
		\end{eqnarray*}
		Let $\phi$ be the composition
		\begin{eqnarray*}
			p_1^*N_1\maps E_Q \maps p_2^*\un_2
		\end{eqnarray*}  
		We define $F=D_0(\phi)$ to be the vanishing scheme of $\phi$, see \cite[Example 4.2.8]{Sernesi2006}. By \emph{loc. cit.}, it is a closed subscheme of $Q$, and thus also projective.
		
		We claim that a morphism $Z\maps Q$ defines an element of $FQ_2$ if and only if it factors through $F$, which proves $FQ_2$ is represented by $F$.
		
		Indeed, by the defining property of the vanishing scheme, $f$ factors through $F$ if and only if $f^*(\phi)=0$, if and only if the kernel of $E_Z\twoheadrightarrow \f_1$ maps to zero under $E_Z\twoheadrightarrow \f_2$. This is precisely the condition to have a nested surjection $E_T\twoheadrightarrow \f_1 \twoheadrightarrow \f_2.$
		
		Since $FQ_2=F$ is a closed subscheme of a projective $S$-scheme, it is itself projective over $S$.
	\end{proof}
	
	\printbibliography
\end{document}